\documentclass[a4paper,11pt,reqno,noindent]{amsart}
\usepackage[centertags]{amsmath} 
\usepackage{amsfonts,amssymb,amsthm}
\usepackage{mathtools}

\usepackage{graphicx}
\usepackage{dsfont}
\usepackage[english]{babel} \usepackage{newlfont} \usepackage{color}
\usepackage[body={15cm,21.5cm},centering]{geometry} \usepackage{fancyhdr}
 \usepackage{esint} 

\usepackage[active]{srcltx}
\usepackage{mathrsfs}

\newtheorem{theorem}{Theorem}[section] 
\newtheorem{lemma}[theorem]{Lemma}
\newtheorem{proposition}[theorem]{Proposition}
\newtheorem{corollary}[theorem]{Corollary}
\newtheorem{definition}[theorem]{Definition}

\theoremstyle{definition} 
\newtheorem{remark}[theorem]{Remark}

\newcommand{\conv}{\mathrm{conv}}

\newcommand{\loc}{\mathrm{loc}}

\newcommand{\cc}{\subset\!\subset}
\newcommand{\dist}{\operatorname{dist}} \newcommand{\supp}{\operatorname{supp}}

\newcommand{\e}{\varepsilon}

\newcommand\ecke{\mathop{\hbox{\vrule height 7pt width .3pt depth 0pt \vrule
height .3pt width 5pt depth 0pt}}\nolimits}

\newcommand{\R}{\mathbb{R}} 
\newcommand{\N}{\mathbb{N}} 
 
\renewcommand{\d}{\mathrm{d}} \renewcommand{\L}{\mathbb{L}}
 \newcommand{\M}{\mathcal{M}}
 \newcommand{\id}{\mathrm{Id}}

 \newcommand\sgn{\mathrm{sgn}}
 
\newcommand{\diam}{\mathrm{diam}\,} 
 \renewcommand{\H}{\mathcal{H}}
 
\renewcommand{\L}{{\mathcal L}}

\newcommand\wsto{\stackrel{*}{\rightharpoonup}}

\renewcommand{\div}{\mathrm{div}\,}

\newcommand{\rid}{{\mathcal R}_{v_0}}
\newcommand{\curl}{\mathrm{curl}\,}

\newcommand{\cof}{\mathrm{cof}\,}

\newcommand{\eps}{\varepsilon}

\newcommand{\Lm}{\mathcal{L}}

\DeclareMathOperator*{\argmax}{arg\,max}
\DeclareMathOperator*{\argmin}{arg\,min}

\makeatletter
\def\widebreve{\mathpalette\wide@breve}
\def\wide@breve#1#2{\sbox\z@{$#1#2$}%
     \mathop{\vbox{\m@th\ialign{##\crcr
\kern0.08em\brevefill#1{0.8\wd\z@}\crcr\noalign{\nointerlineskip}%
                    $\hss#1#2\hss$\crcr}}}\limits}
\def\brevefill#1#2{$\m@th\sbox\tw@{$#1($}%
  \hss\resizebox{#2}{\wd\tw@}{\rotatebox[origin=c]{90}{\upshape(}}\hss$}
\makeatletter

\title{Connecting disclinations by ridges}

\date{\today} 
\author[P.~Gladbach]{Peter Gladbach}
\author[H.~Olbermann]{Heiner Olbermann} 
\address[Peter Gladbach]{Institut f\"ur Angewandte Mathematik, Universit\"at Bonn, 53115 Bonn, Germany}
\address[Heiner Olbermann]{Institut de Recherche en Math\'ematique et Physique, UCLouvain, 1348 Louvain-la-Neuve, Belgium}
\email[Peter Gladbach]{gladbach@iam.uni-bonn.de}
\email[Heiner Olbermann]{heiner.olbermann@uclouvain.be}

\begin{document}

\begin{abstract}
  We consider a thin elastic sheet with a finite number of disclinations in the variational framework of the F\"oppl--von K\'arm\'an approximation. Under the non-physical assumption that the out-of-plane displacement is a convex function, we prove that minimizers display ridges between the  disclinations. We prove the associated energy scaling law with upper and lower bounds that match up  to logarithmic factors in the thickness of the sheet. One of the key estimates in the proof that we consider  of independent interest is a generalization of the monotonicity property of the Monge-Amp\`ere measure.
\end{abstract}

\maketitle

\section{Introduction}

The present article draws its motivation from the mathematical analysis of strongly compressed  thin elastic structures, corresponding to the \emph{crumpling} of paper. This  phenomenon has received a lot of attention in the physics literature over the last  decades, see e.g.~\cite{PhysRevE.71.016612,RevModPhys.79.643,PhysRevLett.87.206105,Lobkovsky01121995,witten1993asymptotic,PhysRevLett.80.2358,Cerda08032005,pocivavsek2008stress}.
Viewing the crumpled configuration as a minimizer of a free energy, one expects from the everyday experience of crumpling paper that the minimizer should be characterized  by a complex network of one-dimensional ridges that meet in vertices, with the elastic energy concentrated in this network.

Closer inspection  shows that elastic energy should be concentrated only in the ridges, the formation of vertices being energetically less costly. As the thickness of the sheet $h$ is sent to 0 with all other parameters in the setting remaining fixed, one expects that a finite number of ridges should carry the elastic energy, which scales with $h^{5/3}$. This conjecture has been supported by the lower bounds proved by Venkataramani in \cite{MR2023444}, where a single ridge is considered as the result of clamping an elastic  sheet  along the lateral boundaries of a bent frame. Venkataramani's lower bound holds in  the F\"oppl--von K\'arm\'an approximation, and was later  upgraded by Conti and Maggi \cite{MR2358334} to a lower bound with the same scaling in the geometrically fully nonlinear setting. The latter authors also provide an upper bound with the same scaling $h^{5/3}$ for a much more general situation,  in the form of an approximation result for piecewise affine isometric maps.

The boundary conditions considered for a single ridge in \cite{MR2358334,MR2023444} are such that any deviation from the sharp ridge satisfying these conditions involves stretching of the sheet, and is hence very costly. This fact is at the heart of the respective proofs.  In the sequel, we will call boundary conditions with this feature \emph{tensile}. 

\medskip

From a mathematical point of view, the next step in the asymptotic analysis would be to prove the \emph{emergence} of ridges without enforcing them by tensile boundary conditions, and an associated lower bound with the expected scaling. For example, only assuming that the elastic sheet is confined by a container whose diameter is much smaller than that of the sheet, a setting that we will refer to as the ``crumpling problem'' from now on, is an extremely difficult and up to now entirely open problem.

One way of understanding the difficulty is the following: If boundary conditions and other constraints that define  the problem allow for a large class of \emph{short} maps (i.e., deformations $y: \Omega\to \R^3$ that satisfy $\nabla y^T\nabla y\leq \id_{2\times 2}$ in the sense of positive definite matrices), then by the famous Nash-Kuiper theorem \cite{MR0065993,MR0075640} there exist very many competitors that are arbitrarily close to isometric immersions. Since the distance of the induced metric from that of an isometric immersion is the leading order term in the typical elastic energy
  \begin{equation}\label{eq:31}
    I_h(y) := \int_\Omega |\nabla y^T\nabla y-\id_{2\times 2}|^2+|\nabla^2 y|^2\d x\,,
  \end{equation}
 one needs  quantitative lower bounds for the bending energy for this large set of competitors. 

\medskip

A simpler setting, which at first sight seems well adapted to investigate the emergence of ridges, is an elastic sheet featuring two \emph{disclinations} (metric defects). The case of one disclination has  been investigated in \cite{Olbreg,Olbdisc,olbermann2018shape}, where, aside from an energy scaling law, it has been shown that minimizers converge to a conical configuration as the thickness of the sheet approaches zero. In the case of two  disclinations, one expects approximately conical shapes in the neighborhood of each of them, and the formation of a ridge between  them.\footnote{This comes with the caveat that the formation of a ridge in such a situation can be avoided, unless  further constraints or boundary conditions are introduced into  the problem, see Section \ref{sec:counterexample} in the appendix. In the statement of Theorem \ref{thm:main}, this additional assumption is given by the Dirichlet boundary values $v=0$ on $\partial\Omega$.} Since one single ridge is expected at a known position, it seems that this setting should be much more accessible than the full crumpling problem as described above. Alas, also in this setting the derivation of a lower bound with the conjectured $h^{5/3}$ scaling seems to be prohibitively difficult. To the best of the authors' knowledge, this also remains an entirely open problem.

\medskip

In the present article we will solve the  problem we have just described under an additional non-physical assumption: Working in the F\"oppl--von K\'arm\'an approximation, which distinguishes between the ``in-plane'' and ``out-of-plane'' components of the deformation. We assume that the latter is \emph{convex}. Under this assumption, we are able to show the emergence of a ridge between two  disclinations, and an associated energy scaling law with upper and lower bounds that match up to logarithmic factors in  $h$.

\medskip

More precisely, we work with a finite number of disclinations, and  show that up to logarithmic factors, the energy of the ridges between them scales with $h^{4/3}$. As is to be expected, the energy under the convexity assumption is larger than without. In Section \ref{sec:ub-non-convex}, we show that an energy bounded by  $Ch^{5/3}$ can be achieved  in the absence of the constraint of convexity of the out-of-plane component. This means that our results constitute a rigorous proof of the nonconvexity of the optimal out-of-plane displacement for small enough $h$ (see Corollary \ref{cor:min-non-convex}).


\newcommand{\bu}{\mathbf{u}}

\subsection{Statement of the main result}
Let $\Omega\subseteq \R^2$ be convex, open, bounded with $C^{1,1}$ boundary, $\{a_1,\dots, a_N\}\subseteq\Omega$ a set of disclination vertices,  $\sigma_i>0$ for $i=1,\dots,N$, and  $\overline \mu:=\sum_{i=1}^N \sigma_i\delta_{a_i}\in \M(\overline \Omega)$ a disclination measure.

  For a convex function $u: \Omega\to \R$, let $\partial^- u(x)$ denote the subdifferential of $u$ at $x$, 
\[
\partial^- u(x) := \{p\in \R^n: p\cdot (y-x)+u(x)\leq u(y)\text{ for all }y \in \Omega\}
\]
and let 
$\mu_u$ denote its Monge-Amp\`ere measure defined for a Borel set $E\subseteq \Omega$ as
\[
\mu_u(E) := \L^2\left(\partial^- u(E)\right)\,,
\]
where $\L^2$ denotes the Lebesgue measure on $\R^2$.
By \cite[Theorem 2.13]{figalli2017monge}, there exists a unique convex solution  $v_0$ of  the Monge-Amp\`ere equation
  \begin{equation}\label{eq: MAD}
  \left\{\begin{split}
    \mu_{v_0}&=\overline \mu \text{ in }\Omega\\
    v_0&=0 \text{ on }\partial\Omega\,.
  \end{split}\right.
\end{equation}
By Lemma \ref{lemma: 2d ridge} below,  $v_0\not\in W^{2,2}(\Omega)$ whenever $N\geq 1$, since a ridge appears between two vertices.

We define the F\"oppl--von K\'arm\'an energy for deformations $(\bu,v)\in W^{1,2}(\Omega;\R^2)\times W^{2,2}(\Omega)$, with the reference configuration given by the ``singular'' configuration $v_0$, 
  \begin{equation}\label{eq:1}
E_h(\bu,v)=\int_\Omega \left|\nabla \bu+\nabla \bu^T+\nabla v\otimes\nabla v-\nabla v_0\otimes\nabla v_0\right|^2+h^2|\nabla^2 v|^2\d x\,.
\end{equation}
This functional can be viewed as a partial linearization of the geometrically nonlinear functional  \eqref{eq:31}, with  metric defects. A rigorous relation between these energies can be established in the sense of $\Gamma$-convergence, see \cite{MR2210909}.
In the definition of the energy \eqref{eq:1}, we have chosen to absorb the in-plane deformation $\bu_0$ that would yield vanishing stress in the reference configuration via $\nabla \bu_0+\nabla \bu_0^T+\nabla v_0\otimes \nabla v_0=0$, which exists, into $\bu$. 
Let
\[
  \mathcal A:=\left\{(\bu,v)\in W^{1,2}(\Omega;\R^2)\times W^{2,2}(\Omega): v\text{ convex, }  v=0 \text{ on }\partial \Omega\right\}\,.
\]
We will prove:

\begin{theorem}
  \label{thm:main}
  There exist $h_0,C>0$ that only depend on $\Omega$ and $\overline\mu$ such that
  \[
    \frac{1}{C}h^{4/3} \left(\log \frac{1}{h}\right)^{-1/2} \leq   \inf_{(\bu,v)\in\mathcal A} E_h(\bu,v)\leq C h^{4/3}
  \]
  for all $h<h_0$.
\end{theorem}

\begin{remark}
  \begin{itemize}
  \item[(i)] The sources of the logarithmic factor in the  lower bound are discussed in Remark  \ref{rem:log_lb} below. For a sketch of the idea of the lower bound, see the beginning of Section \ref{sec:lb}.
\item[(ii)] The boundary conditions $v=0$ on $\partial\Omega$ are necessary in order to ``force'' the emergence of a ridge between the vertices defined by $\overline \mu$. If these boundary conditions are omitted, then it is possible to achieve an energy that scales with $h^2\log\frac{1}{h}$, which is much smaller than $h^{4/3}$. For the construction of the test functions achieving the former, see Section \ref{sec:counterexample} of the appendix. 
\end{itemize}

\end{remark}

\subsection{The method of proof: Gauss curvature and convexity}
As should be clear from the discussion so far, of the two inequalities  in Theorem \ref{thm:main}, the lower bound is the more interesting and difficult one to prove. Its proof is based on a principle  that has already been used  in \cite{Olbreg,Olbdisc,olbermann2018shape,olber2019crump,gladbach2024variational}. In these papers lower bounds for the elastic energy are derived starting from a certain control over the (partially linearized) Gauss curvature of the deformed surface, which is obtained either directly from the constraints in the problem or from energetic considerations. The cited papers do this in settings that produce approximately conical configurations, which are characterized by an elastic energy that scales with $h^2\log \frac{1}{h}$.

\medskip

The geometric idea is that the leading order term in the elastic energy (distance of the induced metric from that of an isometric immersion) controls the Gauss curvature. Knowledge about the Gauss curvature can be translated into  information about the image of the normal map. This information can in turn  be exploited to obtain lower bounds for the second gradient of the deformation. In this article, we work with the partially linearized Gauss curvature, i.e. the Monge-Amp\`ere measure $\mu_v = \det\nabla^2 v$, in order to gain information about the image of $\nabla v$, in order to obtain lower bounds for the bending energy $h^2 \int |\nabla^2 v|^2\,\d x$.

\medskip

If one assumes convexity of the out-of-plane displacement $v$ and hence non-negativity of the Monge-Amp\`ere measure $\mu_v$, this chain of arguments can be made more powerful by noting that knowledge about the Monge-Amp\`ere measure  directly gives some information about the deformation itself via the monotonicity of the subdifferential of convex functions and the Alexandrov maximum principle. In fact, our proof uses as an essential tool the following generalization of the monotonicity property of the Monge-Amp\`ere measure, which we consider of interest in its own right (see Proposition \ref{prop: comparison} below):
If   $\Omega\subseteq \R^n$ is a convex bounded open domain, and $\phi\in C^0(\overline \Omega)$ a non-negative concave function, then whenever $u,v\in C^0(\overline \Omega)$ are convex with $u=v$ on $\partial \Omega$ and $u\geq v$, we have
  \begin{equation}
    \int_\Omega\phi\d \mu_u \leq \int_\Omega\phi\d \mu_v\,.
  \end{equation}

The classical monotonicity property is the special case $\phi=1$. 
The proof  will be given in Section \ref{sec:comparison_principle} below. 

\medskip

The combination of the F\"oppl--von K\'arm\'an approximation with the assumption of convexity of the out-of-plane component has been studied  in \cite{gladbach2024variational}, where a partially clamped elastic sheet in the shape of the sector of a disk has been considered. The method of proof used in that paper is adapted to the particular setting considered there.  The tools we will present here are different and more versatile.

\subsection{Scientific context}

Crumpling is only one of many possible \emph{patterns} that thin elastic sheets may form under the influence of external forces. Explaining these patterns as the result of  energy minimization has proved to be a fruitful principle within the  physics and mathematics community, see e.g.~the overview article by Witten \cite{RevModPhys.79.643} or the book by Audoly and Pomeau \cite{audoly2010elasticity}. In the calculus of variations,  two types of results have emerged: Firstly, the derivation of reduced models of three-dimensional nonlinear elasticity in the limit of small film thickness by means of $\Gamma$-convergence. Notable successes include the derivation of a membrane model by Le Dret and Raoult \cite{le1995nonlinear} as well as  the derivation of nonlinear plate models by Friesecke, James and M\"uller \cite{MR1916989,MR2210909}. Following up on these groundbreaking works, there have also been $\Gamma$-convergence results for shells \cite{MR2796137,MR2731157,MR1988135,MR4076073}. In this context, significant research activity has developed around the derivation of Korn-type inequalities for slender bodies \cite{MR3687880,MR3262603,MR2326996,MR2795715}, which is also of importance for the onset of buckling. Another interesting set of $\Gamma$-convergence results is concerned with curved reference configurations \cite{tobasco2021curvature}.

\medskip

Second, there is the investigation of the qualitative properties of low-energy states in the variational formulation of elasticity. An early instance of this approach in nonlinear elasticity is due to Ball and James \cite{ball1987fine}. A popular way of carrying out such an analysis is by  the derivation of energy scaling laws (and is, of course, the approach followed in the present article), such as as in the work by Kohn and M\"uller \cite{MR1293775,MR1272383}. Since then there have been many successful applications to  phenomena in materials science. In the case of thin elastic sheets, there have been notable results for sheets wrinkling under tension \cite{MR3179665}, under compression \cite{MR1921161,MR3745154}, or when attached to a substrate \cite{kohn2013analysis,MR3646084}. Energy scaling laws for nearly conical deformations of thin elastic sheets have been considered in \cite{2012arXiv1208.4298M,MR3102597}, building on which there have been derivations of  small-deflection limit models in the sense of $\Gamma$-convergence \cite{olbermann2014one,MR3766980}.

\subsection{Notation}

For $A\subseteq\R^n$, we will denote the convex envelope of $A$ as $\conv A$, and the interior of $A$ by $A^\circ$. The characteristic function of $A$ is denoted by $\mathds{1}_A$.  For a finite number of points $a_1,\dots,a_k\in\R^n$, we will write $[a_1a_2\dots a_k]:=\conv \{a_1,\dots,a_k\}$. For $x=(x_1,x_2)\in\R^2$ we write $x^\bot:=(-x_2,x_1)$. 

The $k$-dimensional Hausdorff measure is denoted by $\H^k$, the $n$-dimensional Lebesgue measure by $\L^n$. The set of signed Radon measures on $\Omega\subset\R^n$ is denoted by $\mathcal M(\Omega)$. The subset of non-negative elements of $\mathcal M(\Omega)$ is denoted by $\mathcal M_+(\Omega)$. For every $\mu\in \mathcal M(\Omega)$ there exists a decomposition $\mu=\mu_+-\mu_-$ with $\mu_{\pm}\in\mathcal M_+(\Omega)$. We write $|\mu|:=\mu_++\mu_-$. For a Radon measure $\mu$ and a $|\mu|$-integrable  function $f$, the signed Radon measure $\mu\ecke f$ is defined by $\mu\ecke f(A):=\int_A f\d \mu$ for $A\subseteq\R^n$ $\mu$-measurable. For sequences $\mu_k\in \mathcal M(\Omega)$ and $\mu\in \mathcal M(\Omega)$ such that
  \begin{equation}\label{eq:33}
\lim_{k\to\infty}\int_\Omega\varphi\d\mu_k\to \int_\Omega\varphi\d\mu\quad\text{ for all } \varphi\in C^0_c(\Omega)\,,
\end{equation}
we say that $\mu_k$ converges weakly-* to $\mu$ and write $\mu_k\wsto \mu$. We note that when $\mu_k,\mu\in \mathcal M_+(\Omega)$, then \eqref{eq:33} is equivalent to 
\[
  \begin{split}
\mu(U)&\leq \liminf_{k\to\infty}\mu_k(U) \quad \text{ for all open sets }U\subseteq\Omega\\
\text{ and } \mu(K)&\geq \limsup_{k\to\infty}\mu_k(K) \quad \text{ for all compact sets }K\subseteq\Omega\,.
\end{split}
\]
For  matrices $A,B\in \R^{n\times n}$, we define the dot product $A:B=\mathrm{Tr}(A^TB)$, where $\mathrm{Tr}$ denotes the trace. The matrix  $\cof A\in \R^{n\times n}$ denotes the matrix of cofactors of $A$, $(\cof A)_{ij}=(-1)^{i+j}\det \hat A_{ij}$, where $\hat A_{ij}$ is the $(n-1)\times(n-1)$ matrix obtained from $A$ by deleting the $i$-th row and the $j$-th column. 

The symbol ``$C$'' will be used as follows: An inequality such as $f\leq Cg$ is always to be understood as  the statement ``there exists a constant $C>0$ such that $f\leq Cg$''; sometimes we will put the full statement for clarity, in other places we just write the inequality in an attempt to increase readability. When we write $C(a,b,\dots)$, then $C$ may depend on $a,b,\dots$. Dependence of $C$ may be omitted from the notation when it is clear from the context.  Instead of $f\leq Cg$, we also write $f\lesssim g$. In chains of inequalities, the value of $C$ may change at each time that the symbol is used. 

\subsection{Plan of the paper} In Section \ref{sec:MAanalysis}, we derive some results in convex analysis; in particular we prove the generalized monotonicity property of the Monge-Amp\`ere measure and we  study the piecewise affine-conical structure of the solution $v_0$ of \eqref{eq: MAD}. Using this structure we may then construct the upper bound in Theorem \ref{thm:main} in Section \ref{sec:ub-convex}. In Section \ref{sec:ub-non-convex} we provide the upper bound construction for the situation when the convexity constraint for the out-of-plane component is discarded, achieving a lower energy. Section \ref{sec:lb} contains the proof of the lower bound. It starts off with a rough overview of the proof, followed up by the details in Subsections \ref{sec:suboptimal}-\ref{sec:lower_bound}. In the appendix, we show that when no boundary conditions are imposed on $v$, an energy that scales with $h^2\log\frac{1}{h}$ can be achieved (much lower than what we found in Theorem \ref{thm:main}), with out-of-plane diplacements that do not display any ridges.

\subsection*{Acknowledgments}
The authors would like to thank Ian Tobasco for pointing out that equality holds in the situation of Lemma \ref{lem:W-22L2equality} (instead of just an inequality). This has led  to an improvement of the upper bound construction.

\section{Some convex analysis}

\label{sec:MAanalysis}
Let $\Omega\subseteq\R^n$ be open, convex and bounded. For a convex function $u:\overline\Omega\to\R$,  
the notions of subgradient $\partial^-u$ and Monge-Amp\`ere measure that we have introduced above for the special case $n=2$ have their obvious analogue in the general case. We recall that $\mu_u\in\M(\Omega)$ is a Radon measure, and the following elementary fact concerning subgradients, that we will repeatedly use in our proofs below without further comment: 
\[
 p\in \partial^-u(x) \quad\Leftrightarrow \quad p\cdot x-u(x)\geq p\cdot y-u(y) \quad \text{ for  all } y\in\overline\Omega\,.
  \]


\medskip

As is well known, any convex function $u:\overline{\Omega}\to\R$ can be written as the supremum of affine functions that are smaller than $u$,
\[
u(x)=\sup\{a(x):a\text{ affine },\,a\leq u\}\,.
\]
For a possibly non-convex function $u:\overline\Omega\to\R\cup \{+\infty\}$ we define the convex envelope as 
\[
u^c (x):=\sup\{v(x): v\text{ convex }, \,v\leq u\}\,.
\]
Combining with the previous identity, we have  
\[
u^c (x)=\sup\{v(x): a\text{ affine }, \,a\leq u\}\,.
\]
For $E\subseteq A$ and $u:A\to\R$, we define the restriction $R_Eu:A\to \R\cup\{+\infty\}$,
\[
R_E u (x):=\begin{cases} u(x)& \text{ if }x\in E\\
+\infty&\text{ else. }\end{cases}
\]
We define the lifting of a convex function on a bounded set:
\begin{definition}
  Let $K\subseteq \R^n$ be compact, $\Omega:=(\conv K)^\circ$, and $u\in C^0(\overline\Omega)$.
  We define the lifting of $u$ on $K$ as $L_K u\in C^0(\overline{\Omega})$,
  \[
  L_K u := (R_K u)^c\,.
  \]
Here it is understood that the supremum in the definition of $(R_K u)^c$ is taken over convex functions $v:\overline\Omega\to\R$ that satisfy $v\leq R_K u$. 
\end{definition}

We note that $L_K u$ is  a well-defined convex function $\overline{\Omega}\to \R$. If $u$ itself is convex, then   $L_K u \geq u$ in $\conv K$, and $L_K u = u$ on $K$. As a consequence of the above identities, 
    \begin{equation}\label{eq:5}
      L_K u(x) = \sup\{ a(x)\,: a\text{ affine },\,a\leq u\text{ on }K\}\,.
\end{equation}

The lifting $L_Ku$ has the property that its Monge-Amp\`ere measure is concentrated on the set $K$, as we will show in the upcoming lemma. 

\begin{lemma}\label{lem: lifting}
Let $K\subseteq\R^n$ be compact, $\Omega=(\conv K)^\circ$, $u\in C^0(\overline{\Omega})$ convex,  $w=L_K u$, and  $p\in \partial^- w(\Omega)$. Then  all extreme points of ${(\partial^-w)^{-1}(p)}$ are contained in  $K$. 
In addition, we have
   \[
\mu_{w}(\Omega \setminus K) = 0\,.
\]
\end{lemma}

\begin{proof}
Assume there are $x\in \Omega$, $p\in \partial^- w(x)$ and an extreme point $y$ of ${(\partial^-w)^{-1}(p)}$ with $y\not\in K$.  Then $a(z) = u(x) + p\cdot (z-x)$ defines the affine function satisfying  $(\partial^-w)^{-1}(p)=\{z:a(z)=w(z)\}$.
Since $y$ is an extreme point of the convex set $(\partial^-w)^{-1}(p)$, there exists a closed half-space $H$ such that 
\begin{equation}\label{eq:35}
  \{y\}=H\cap (\partial^-w)^{-1}(p)\,.
\end{equation}
We claim that there exists $\delta>0$ and an open neighborhood $U$ of $K\cap H$ such that 
\[
w-a\geq \delta\quad \text{ on } U\,.
\]
Indeed, suppose this were not the case. Then we may set 
\[
\delta_k=k^{-1}\,,\quad U_k=\{x':\dist(x',K\cap H)<k^{-1}\}
\]
for $k\in\N$
and obtain $x_k\in U_k$ such that $(w-a)(x_k)< k^{-1}$. By compactness of $\overline{U_1}$, we obtain a convergent subsequence $x_{k_l}\to\bar x$ with $\bar x\in K\cap H$ and $(w-a)(\bar x)=0$. This is a contradiction to \eqref{eq:35}, and hence proves the existence of $U,\delta$ as claimed.

\medskip

Now we will construct an affine function $b$ with $b(y)>0$ and $a+b\leq w$ on $K$. Let $\nu$ be the normal to $\partial H$ pointing in the direction of $H$. Then with $\e,\e'>0$ to be chosen later, we may set
\[
b(x')=\e+\e' \nu\cdot(x'-y)\,.
\]
Clearly, $b(y)>0$. Also, for $x'\in K\cap U$, $(a+b)(x')\leq w$ holds for $\e,\e'$ chosen suitably small by continuity. Since $K\setminus U$ has positive distance from $\partial H$, we may achieve $(a+b)\leq w$ on $K\setminus U$ by choosing $\frac{\e'}{\e}$ small enough. 

\medskip

Recalling $w=u$ on $K$, we obtain that 
\[
a+b\leq u \text{ on } K\,.
\]
Using \eqref{eq:5}, it follows $w(y)\geq a(y)+ b(y) > a(y) = w(y)$, a contradiction. This shows  $y\in K$.

\medskip

To see that $\mu_{w}(\Omega\setminus K) = 0$, we define the set 
\[
N := \{p\in \R^n\,:\,\#(\partial^- w)^{-1}(p)>1\}\,.\]
Denoting by $w^*$ the Legendre transform of $w$, we have that $(\partial^- w)^{-1}=\partial^- w^*$ (see \cite[Theorem 23.5]{MR0274683}),  and hence $N= \{p\in \R^n\,:\,\#\partial^- w^*(p)>1\}$. Since $w^*$ is locally Lipschitz, and $\#\partial^- w^*(p)>1$ implies non-differentiability of $w^*$ in $p$, we have that $\L^n(N)=0$ by Rademacher's Theorem. By what we have already shown, if $x\in (\partial^-w)^{-1}(p)\setminus K$, then $x$ is not an extreme point of ${(\partial^-w)^{-1}(p)}$, which of course implies $\#(\partial^-w)^{-1}(p)>1$, i.e., $p\in N$. Thus we have $\partial^- w (\Omega\setminus K) \subseteq N$ which in turn yields $\mu_{w}(\Omega\setminus K) = \L^n(\partial^- w (\Omega \setminus K))\leq \L^n(N) = 0$.
\end{proof}

The following lemma describes an approximation of convex functions from above, which we will use repeatedly  in our proofs below. 

\begin{lemma}
\label{lem:convex_smoothing}
  Let  $u\in W^{1,\infty}(\R^n)$ be a convex function, $\e>0$, and  
  \begin{equation}\label{eq:23}
    P_\e u(x):=\inf\left\{q(x): q\in C^2(\Omega), \nabla^2 q\leq \e^{-1}\id, q\geq u \right\}\quad\text{ for }x\in\R^n\,,
  \end{equation}
where the inequality $\nabla^2q\leq \e^{-1}\id$ is meant in the sense of positive definite matrices. 
Then $P_\e u$ is convex,  with $\|\nabla^2 P_\e u\|_{L^\infty(\R^n)} \leq \frac{1}{\e}$, and writing  $c=4\|\nabla u\|_{L^\infty(\R^n)}$ we have that $u=P_\e u$ on the set
\[
\left\{x\in \R^n: u\in W^{2,\infty}(B(x,c\e)), \|\nabla^2 u\|_{L^\infty(B(x,c\e))}\leq \frac{1}{\e}\right\}\,.
\]
\end{lemma}

\begin{proof}

First we note that the right hand side in \eqref{eq:23} can be slightly changed, replacing the inequality $\nabla^2 q\leq \e^{-1}\id$ by an equality, and that the infimum is achieved:
  \begin{equation}\label{eq:28}
    P_\e u(x)=\min\left\{q(x): q\in C^2(\R^n), \nabla^2 q= \e^{-1}\id, q\geq u \right\}\,.
  \end{equation}
Here the existence of the minimum for every $x$ follows from a straightforward compactness argument. 

\medskip 

Now we show that $P_\e u$ is convex. Let $x_0,x_1\in \R^n$ and $x = (1-\lambda)x_0 + \lambda x_1$ for some $\lambda\in[0,1]$. For $i=0,1$ let $Q_i$ be quadratic functions with $\nabla^2 Q_i = \frac{\id}{\e}$, $Q_i\geq u$ on $\Omega$, and $Q_i(x_i)= u(x_i)$.

Define 
\[
\tilde Q:=(\min(Q_0,Q_1))^c\,.
\]
Since $u$ is convex, we have $\tilde Q\geq u$. We note that $\tilde Q\in W^{2,\infty}(\R^n)$ with $\nabla^2 \tilde Q \leq \frac{\id}{\e}$ almost everywhere. This may either be deduced by computing  $\tilde Q$ explicitly, or from the general regularity results of convex envelopes from \cite{MR1868942,griewank1990smoothness}. 


  Define the quadratic function $Q(y):= \tilde Q(x) + \nabla \tilde Q(x)\cdot(y-x) + \frac{|y-x|^2}{2\e}$. Then
 $Q \geq \tilde Q \geq u$ on $\Omega$ and
 \[
 P_\e u(x) \leq Q(x) = \tilde Q(x) \leq (1-\lambda)Q_0(x_0) + \lambda Q_1(x_1)\leq (1-\lambda) P_\e u(x_0)+ \lambda P_\e u(x_1) \,.
\]
This proves the convexity of  $P_\e u$.

To see that $\nabla^2 P_\e u \leq \frac{\id}{\e}$, we use   that for every $x\in \R^n$, there is a quadratic function $Q$ with $\nabla^2 Q = \frac{\id}{\e}$, $Q\geq P_\e u$ and $Q(x) = P_\e u (x)$ as we have noted in \eqref{eq:28}. Since $P_\e u$ is also convex, this implies that $P_\e u$ is everywhere differentiable and that $\nabla P_\e u$ is Lipschitz, with $0\leq \nabla^2 P_\e u \leq \frac{\id}{\e}$ almost everywhere.

Finally, assume that $x\in\R^n$ with $\|\nabla^2 u\|_{L^\infty(B(x,4\|\nabla u\|_\infty\e))}\leq \frac{1}{\e}$. Let $Q$ be chosen as above.  By the assumed bound on $\|\nabla^2u\|_{L^\infty}$, we have $u\leq Q$ in $\Omega \cap  B(x,4\|\nabla u\|_\infty\e)$. For $y\in \Omega \setminus  B(x,4\|\nabla u\|_\infty \e)$, we have by Lipschitz continuity $u(y)\leq u(x) + \|\nabla u\|_\infty|y-x|$, whereas
\[
  \begin{split}
Q(y) &\geq  u(x) - \|\nabla u\|_\infty|y-x| + \frac{|y-x|^2}{2\e}\\
& \geq u(x) - \|\nabla u\|_\infty|y-x| + \frac{|y-x|4\|\nabla u\|_\infty \e}{2\e}\\
& \geq u(y).
\end{split}
\] 
Hence  $Q\geq u$ on $\Omega$ and $u(x) = Q(x)$, which implies  $P_\e u (x) = u(x)$, completing the proof.
\end{proof}

\subsection{Monotonicity properties for convex functions}
\label{sec:comparison_principle}

We now state several invariances of Monge-Amp\`ere measures:
\begin{lemma}\label{lem:invariance}
Let $\Omega \subseteq \R^n$ be convex, open, and bounded. Let $u,v\in C^0(\R^n)$ be convex functions with $u=v$ in $\R^n\setminus \Omega$. Then
\begin{itemize}
  \item [(i)] $\partial^- u(\overline \Omega) = \partial^- v(\overline \Omega)$. In particular, $\mu_u(\overline\Omega) = \mu_v(\overline\Omega)$.
  \item [(ii)] $\int_{\overline\Omega}x\d \mu_u = \int_{\overline\Omega}x\d \mu_v$.
\end{itemize}
\end{lemma}

\begin{proof}
To prove (i), let  $p\in \partial^-v(\overline\Omega)$. Then the concave function $x\mapsto p\cdot x-v(x)$  achieves its maximum on $\overline\Omega$, and hence  the maximum of that function on $\R^n\setminus\Omega$ is achieved on $\partial\Omega$, 
\[
\max_{x\in \R^n\setminus \Omega} p\cdot x - v(x) = \max_{x\in \partial\Omega} p\cdot x - v(x).  
\]
Since $u=v$ on $\R^n\setminus \Omega$, this identity reads 
\[
\max_{x\in \R^n\setminus \Omega} p\cdot x - u(x) = \max_{x\in \partial\Omega} p\cdot x - u(x)\,,
\]
which implies
  $p\in \partial^-u(\overline\Omega)$. Reversing the roles of $u,v$ yields the claim.

  To prove (ii), assume first that $u, v\in W^{2,\infty}(\R^n)$. 
Let $i\in \{1,\dots,n\}$. Writing $w=v-u$, we have
  \[
    \begin{split}
      \left.\frac{\d}{\d t}\right|_{t=0}  \int_\Omega x_i\d \mu_{u+t(v-u)} 
      &=\left.\frac{\d}{\d t}\right|_{t=0}  \int_\Omega x_i\det \nabla^2 (u+tw) \d x\\ 
      &= \int_\Omega x_i\cof \nabla^2 u(x) : \nabla^2 w(x) \,\d x\,.
    \end{split}
  \]
We will show that the right hand side vanishes, which is sufficient to prove our claim. 
By the distributional Piola identity $\div\cof \nabla^2u=0$, we get that 
\[
\div\left(x_i\,\cof \nabla^2 u\nabla w\right), \,\div\left(w \,\cof \nabla^2 u \,e_i\, \right)\in L^2(\Omega)\,,
\]
and hence 
\[
x_i\,\cof\nabla^2u\nabla w,\,\, w\, \cof \nabla^2 u e_i\,\in H(\div;\Omega):=\{f\in L^2(\Omega;\R^n):\div f\in L^2(\Omega)\}\,.
\]
It is well known that $C^\infty(\overline\Omega)$ is dense in $H(\div;\Omega)$ and that a continuous normal trace operator $H(\div;\Omega)\to H^{1/2}(\partial \Omega)$ exists (see e.g.~\cite[Chapter 20]{tartar2007introduction}). In particular, one may integrate by parts as follows, 
  \[
  \begin{split}
   \int_\Omega &x_i\cof \nabla^2 u(x) : \nabla^2 w(x) \,\d x\\
    &=  -\int_\Omega  \cof \nabla^2 u(x) : \nabla w(x)\otimes e_i\,\d x + \int_{\partial \Omega} x_i\cof \nabla^2 u(x): \underbrace{\nabla w(x)}_{=0}\otimes n(x) b\,\d\H^{n-1}\\
    &=  -\int_{\partial \Omega} \underbrace{w(x)}_{=0}\cof \nabla^2 u(x):e_i\otimes n(x)\,\d\H^{n-1}\\
&=0\,.
  \end{split}  
  \]
  Here we have  used the fact that $w=\nabla w=0$ on $\partial \Omega$. This completes the proof for the case $u,v\in W^{2,\infty}_{\mathrm{loc.}}(\R^n)$. We note in passing that for this case, the statement and proof hold true also for non-convex $u,v$.



If $u,v$ are only continuous, we approximate both $u$ and $v$ using Lemma \ref{lem:convex_smoothing}, yielding $u_\eps,v_\eps\in W^{2,\infty}_{\mathrm{loc.}}(\R^n)$ with $u_\eps = v_\eps$ outside of $B(\Omega,\rho(\eps))$, with $\rho(\eps) \to 0$ as $\e\to 0$, and $u_\e\to u$, $v_\e\to v$ locally uniformly. From this convergence follows the weak-* convergence of  the Monge-Amp\`ere measures $\mu_{u_\e},\mu_{v_\e}$  to $\mu_u,\mu_v$ respectively (see e.g.~\cite[Proposition 2.6]{figalli2017monge}). Let us fix  $i\in \{1,\dots,n\}$, and  choose $a>0$ such that $\phi(x):=a+x_i\geq 0$ on $\overline\Omega$. Multiplying the convergent sequences by $\phi$ we get
\[
\mu_{v_\e}\ecke \phi\wsto \mu_v\ecke \phi,\quad \mu_{u_\e}\ecke \phi\wsto \mu_u\ecke \phi\,.
\]
We choose a non-increasing sequence of open convex  bounded sets $\Omega_\delta$ with $\Omega\cc\Omega_\delta$ and  $\cap_{\delta>0}\Omega_\delta=\overline\Omega$. Now using the approximation  properties of Radon measures (see e.g.~\cite[Chapter 1]{evans1992measure}), we obtain the following chain of inequalities, 
\[
  \begin{split}
    \mu_{u}\ecke\phi(\overline\Omega)&=\inf_{\delta>0}\mu_{u}\ecke\phi(\Omega_\delta)\\
&\leq\inf_{\delta>0}\liminf_{\e\to 0 }\mu_{u_\e}\ecke \phi(\Omega_\delta)\\
&=\inf_{\delta>0}\liminf_{\e\to 0 }\mu_{v_\e}\ecke \phi(\Omega_\delta)\\
&\leq \inf_{\delta>0} \limsup_{\e\to 0 }\mu_{v_\e}\ecke \phi(\overline{\Omega_\delta})\\
&\leq \inf_{\delta>0} \mu_{v}\ecke \phi(\overline{\Omega_\delta})\\
&= \mu_{v}\ecke \phi(\overline{\Omega})\,.
  \end{split}
\]
Reversing the roles of $u$ and $v$ yields the equality $\mu_{u}\ecke\phi(\overline\Omega)=\mu_{v}\ecke \phi(\overline{\Omega})$, and hence, after subtracting $a \mu_u(\overline\Omega)=a\mu_v(\overline\Omega)$ from this equality, we obtain our claim.
\end{proof}

\begin{remark}
The equality 
\[
\left.\frac{\d}{\d t}\right|_{t=0}\int_\Omega x_i\det\nabla^2(u+tw)\d x=0 \quad\text{ for all }w\in W^{2,\infty}_0(\Omega)\,,
\]
which we have proved above, means of course that the first moment of the Monge-Amp\`ere measure $\int x \det\nabla^2u\d x$ is a \emph{null Lagrangian}.
  To the best of the authors' knowledge, this fact has gone unnoticed in the literature so far.
\end{remark}

As a consequence of the previous lemma, we obtain the following local monotonicity property:

\begin{proposition}\label{prop: comparison}
  Let $\Omega \subseteq \R^n$ be convex, open, and bounded, $u,v\in W^{1,\infty}( \Omega)$  convex with $u=v$ on $\partial \Omega$, $u \geq v$ in $\Omega$,
      and $\phi\in C^0(\overline \Omega)$  concave and non-negative. Then
    \begin{equation}\label{eq: concave comparison}
      \int_\Omega\phi\d \mu_u \leq \int_\Omega\phi\d \mu_v\,.
    \end{equation}
\end{proposition}

\begin{proof}
\emph{Step 1.}
First we suppose that for an arbitrary set of points $\{a_1,\dots,a_M\}\subseteq \Omega$, $u,v$ are of the form $u=L_Ku,v=L_Kv$, where $K=\partial \Omega \cup \{a_1,\ldots,a_M\}$, and $u(a_j)=v(a_j)$ for $j=2,\ldots,M$.
We extend $u$ and $v$ to convex functions $U,V\in W^{1,\infty}(\R^n)$ respectively by
\[
U(x) := \min_{y\in\overline\Omega} u(y) + L|x-y|, \quad V(x) := \min_{y\in\overline\Omega} v(y) +L|x-y|,  
\]
where $L\geq \max(\|\nabla u\|_\infty,\|\nabla v\|_\infty)$ is an upper bound for the Lipschitz constants of $u,v$. 

Recalling that $p\in \partial^-u(a_1)$ if and only if the function $x\mapsto p\cdot x-u(x)$ attains its maximum over $K$ in $a_1$, and $v(a_1)\leq u(a_1)$, $v=u$ on $K\setminus\{a_1\}$, we have that
\[
\partial^-U(a_1)=\partial^- u(a_1) \subseteq \partial^- v(a_1)=\partial^-V(a_1)\,.
\]
Furthermore, since  we have $\partial^-U(\overline\Omega) = \partial^-V(\overline \Omega)$ by Lemma \ref{lem:invariance} (i), 
\[
\partial^-U(y) = \partial^-V(y) \cup (\partial^- U (y)\cap \partial^- v(a_1))\text{ for all }y\in K\setminus \{a_1\}.
\]
In other words, only the excess subgradient $\partial^-v(a_1)\setminus \partial^-u(a_1)$ is redistributed to $K\setminus \{a_1\}$, and
\[
  \begin{split}
0&\leq \mu_V(\{a_1\})-\mu_U(\{a_1\})\\
0&\geq \mu_V-\mu_U\quad\text{ on } K\setminus\{a_1\}\,.
\end{split}
\]
 In particular, the latter implies $\mu_V\leq\mu_U$ on $\partial\Omega$. Now let $b:\Omega\to\R$ be an affine function with $b(a_1)=\phi(a_1)$ and $b\geq \phi$. Then
\begin{align*}
\int_\Omega\phi\d (\mu_v- \mu_u) &=
\int_\Omega \phi \d(\mu_V- \mu_U) \\
 &\geq  \int_{\overline\Omega} \phi \d(\mu_V- \mu_U) \\
&=(\mu_V-\mu_U)\ecke \phi(\{a_1\})+(\mu_V-\mu_U)\ecke \phi(K\setminus\{a_1\})\\
&\geq(\mu_V-\mu_U)\ecke b(\{a_1\})+(\mu_V-\mu_U)\ecke b(K\setminus\{a_1\})\\
&=  \int_{\overline\Omega}b\,\d( \mu_V - \mu_U)\\
& = 0,
\end{align*}
where we used Lemma \ref{lem:invariance} (ii) in the last equality.

\emph{Step 2.} Now let us suppose  $u,v$ are of the form $u=L_Ku,v=L_Kv$, with $K=\partial \Omega \cup \{a_1,\ldots,a_M\}$. The inequality \eqref{eq: concave comparison} follows by applying Step 1 $M$ times.

\emph{Step 3.} Now let  $u,v\in W^{1,\infty}(\overline\Omega)$ be as in the statement of the current proposition. Let $\{a_k\,:\,k\in\N\}\subseteq \Omega$ be a dense subset of $\Omega$. For $k\in \N$, set $u_k = L_{\partial \Omega \cup \{a_1,\ldots,a_k\}}u$, $v_k = L_{\partial \Omega \cup \{a_1,\ldots,a_k\}}v$. 

We have that $u_k\to u$, $v_k\to v$ uniformly, which implies 
\begin{equation}\label{eq:32}
\begin{split}
\mu_{v_k}&\wsto \mu_v\,,\quad\mu_{u_k}\wsto \mu_u\,,\\
\mu_{v_k}\ecke \phi&\wsto \mu_v\ecke \phi\,, \quad\mu_{u_k}\ecke \phi\wsto \mu_u\ecke \phi\,.
\end{split}
\end{equation}
By the classical monotonicity of the subdifferential (see \cite[Lemma 2.7]{figalli2017monge}), 
\[
\mu_{v_k}(\Omega)\leq\mu_v(\Omega)\,,\quad\mu_{u_k}(\Omega)\leq\mu_u(\Omega)\quad\text{ for all }k\in\N\,.
\]
By \eqref{eq:32},
\[
\liminf_{k\to\infty}\mu_{v_k}(\Omega)\geq \mu_v(\Omega)\,,\quad
\liminf_{k\to\infty}\mu_{u_k}(\Omega)\geq \mu_u(\Omega)\,.
\]
Thus 
\[
\lim_{k\to\infty}\mu_{v_k}(\Omega)= \mu_v(\Omega)\,,\quad\lim_{k\to\infty}\mu_{u_k}(\Omega)= \mu_u(\Omega)\,,
\]
which implies
\[
\lim_{k\to\infty}\mu_{v_k}\ecke\phi(\Omega)= \mu_v\ecke\phi(\Omega)\,,\quad\lim_{k\to\infty}\mu_{u_k}\ecke\phi(\Omega)= \mu_u\ecke\phi(\Omega)\,,
\]
and the claim of the proposition follows from the previous step.
\end{proof}

In order to substantiate our claim from the introduction that Proposition \ref{prop: comparison} is of independent interest,  we note in passing that as a direct byproduct, we obtain the following  generalization of the well-known Alexandrov-Bakel'man-Pucci maximum principle (see e.g.~\cite{caffarelli1995fully} for the classical statement).

\begin{corollary}
\label{cor:ABPgen}
  Let $\Omega\subseteq\R^n$ be convex, open and bounded, let $u\in C^2(\Omega)\cap C^0(\overline\Omega)$ satisfy the linear differential inequality
\[Lu=\sum_{i=1}^na^{ij}\partial_{x_i}\partial_{x_j}u\leq g \text{ in }\Omega\]
where $a^{ij}:\Omega\to\R$, $i,j=1,\dots,n$ are coefficients such that $A=(a^{ij})_{i,j=1,\dots,n}>0$ in the sense of positive definite matrices. Let $g\in L^n(\Omega)$, and suppose that $\phi\in C^0(\overline \Omega)$ is a non-negative concave function. Then 
\[
 u(x)\geq \min_{\partial\Omega}u-\frac{\diam(\Omega)}{n(\omega_n\phi(x))^{1/n}}\left\|\left(\frac{\phi}{\det A}\right)^{1/n}g\right\|_{L^n(\Gamma^-)}\quad \text{ for }x\in\Omega\,,
\]
where $\Gamma^-$ is the set where $u$ agrees with its convex envelope.
\end{corollary}

\begin{proof}
Let $u^c$ denote the convex envelope of $u$, and ${\hat u}:=L_{\{x\}\cup\partial\Omega}u$. We note
  \begin{equation}\label{eq:19}
    \omega_n \left(\frac{|u(x)-\min_{\partial \Omega} u|}{\diam(\Omega)}\right)^n\leq \mu_{{\hat u}}(\{x\})\,.
  \end{equation}
By Proposition \ref{prop: comparison}, we have
  \begin{equation}\label{eq:20}
    \phi(x)\mu_{{\hat u}}(\{x\})=\int \phi \d\mu_{{\hat u}}\leq  \int \phi\d\mu_{ u^c}=\int_{\Gamma^-}\phi\det \nabla^2 u\d x\,.
  \end{equation}
Furthermore,  
  \begin{equation}\label{eq:21}
    \det A\det \nabla^2u \leq \left(\frac{\mathrm{Tr} A\nabla^2 u}{n}   \right)^n\,\,.
  \end{equation}
Combining the  inequalities \eqref{eq:19}-\eqref{eq:21} and reordering yields the claim.
\end{proof}

The following lemma is a slightly generalized  version of the Alexandrov maximum principle \cite[Theorem 2.8]{figalli2017monge}:

\begin{lemma}\label{lemma: Alexandrov}
Let $\Omega \subseteq \R^n$ be open, convex, bounded, $U\subseteq\Omega$ open and  $u\in C^0(\overline \Omega)$ convex. Then for every $x\in U$
\begin{equation}\label{eq: Lu - u}
\frac{L_{\partial U}u(x) - u(x)}{\diam(U)} \leq C(n) \mu_u(U)^{1/n}.
\end{equation}
\end{lemma}

\begin{proof}
We define the auxiliary function $v := L_{\partial U \cup \{x\}} u$. We will show that
\begin{equation}\label{eq: subgradient inclusion}
\L^n(\partial^- v(x)) \geq C(n) \frac{|L_{\partial U} u(x) - u(x)|^n}{\diam(U)^{n}}.
\end{equation}

By the monotonicity of the subdifferential (see \cite[Lemma 2.7]{figalli2017monge}),
\[
  \mu_u(U)  \geq \mu_v(U) = \L^n(\partial^- v(x)) \geq C(n) \frac{|L_{\partial U} u(x) - u(x)|^n}{\diam(U)^{n}}\,,
\]
which after rearranging yields \eqref{eq: Lu - u}.

Now to show \eqref{eq: subgradient inclusion}:

Let $p\in\partial^- L_{\partial U} u (x)$. If $q\in \overline{B\left(p,\frac{L_{\partial U}u(x)}{\diam(U)}\right)}$, then
\begin{align*}
q\cdot x - u(x) = & p\cdot x - L_{\partial U}u(x) + (q-p)\cdot x + L_{\partial U}u(x) - u(x) \\
\geq & \max_{z\in \partial U} p\cdot z - L_{\partial U}u(z) + (q-p)\cdot x + L_{\partial U}u(x) - u(x)\\
\geq & \max_{z\in \partial U} q\cdot z - u(z) - |q-p| |x-z| + L_{\partial U}u(x) - u(x)\\
\geq & \max_{z\in \partial U} q\cdot z - u(z).
\end{align*}

This shows that $\overline{B\left(p,\frac{L_{\partial U}u(x)}{\diam(U)}\right)} \subseteq \partial^- v(x)$.
\end{proof}

As a consequence we get the following identity of the solution of \eqref{eq: MAD}.

\begin{lemma}
\label{lem:v0LKv0eq}
  Let $v_0$ be the solution of \eqref{eq: MAD}, and $K:=\partial\Omega\cup\{a_1,\dots,a_N\}$. Then $v_0=L_Kv_0$.
\end{lemma}
\begin{proof}
  The inequality $v_0\leq L_K v_0$ follows directly from the definition of $L_Kv_0$. Assume $v_0(x)<L_Kv(x)$ for some $x\in \Omega\setminus\{a_1,\dots,a_N\}$. By Lemma \ref{lemma: Alexandrov} with $U=\Omega\setminus\{a_1,\dots,a_N\}$, we obtain $\mu_{L_Kv_0}(U)>0$, and hence a contradiction. 
\end{proof}

\subsection{Analysis of the singular Monge-Amp\`ere equation \eqref{eq: MAD}}

We will now  treat the special case  $\Omega\subseteq \R^2$  (still supposing it to be open, convex and bounded). Let $\overline \mu = \sum_{i=1}^N \sigma_i \delta_{a_i}\in \M_+(\Omega)$, with $\sigma_i>0$, $\{a_1,\dots,a_N\}\subseteq \Omega$ and $N\geq 1$. We   characterize 
 the solution $v_0$ to the Monge-Amp\`ere equation with right hand side $\overline\mu$, see \eqref{eq: MAD}.  In particular, Lemma \ref{lemma: 2d ridge} below  will show that   $\nabla v_0\in BV(\Omega;\R^2)$, and that its jump set $J_{\nabla v_0}$ is contained in the set of line segments connecting the  vertices $\{a_1,\dots,a_N\}$. First we state and prove an auxiliary lemma:

 \begin{lemma}
\label{lem:v0aux}
  Let $v_0\in C^0(\overline\Omega)$ be the solution to \eqref{eq: MAD}. Suppose that $p\in \partial^-v_0(\Omega)$, $[a_1a_2]\subseteq (\partial^-v_0)^{-1}(p)$, and 
  \begin{equation}\label{eq:4}
    (\partial^-v_0)^{-1}(p)\cap \{x: (x-a_1)\cdot e_{12}^\bot>0\}=\emptyset\,,
  \end{equation}
where
$e_{12}=\frac{a_2-a_1}{|a_2-a_1|}$. Then there exists $\e_0>0$ such that $p+\e e_{12}^\bot\in \partial^-v_0(x)$ for every $x\in [a_1a_2]$ and every $\e\in(0,\e_0)$.
 \end{lemma}

 \begin{proof}
Let $x\in [a_1a_2]$.   The supporting function 
\[
A(z)=v_0(x)+p\cdot (z-x)=v_0(a_1)+p\cdot(z-a_1)
\]
satisfies $\{A=v_0\}=(\partial^-v_0)^{-1}(p)$. There exists $r>0$ such that
\[
v_0(a)-A(a)>r \quad \text{ for all } a\in\{a_1,\dots,a_N\}\setminus (\partial^-v_0)^{-1}(p)\,.
\]
Writing $L:=(\partial^-v_0)^{-1}(p)\cap\partial\Omega$, the assumption \eqref{eq:4} implies that
there exists $s>0$ with
\[
e_{12}^\bot\cdot (z-a_1)<-s \quad\text{ for  } z\in L\,.
\] 
Furthermore   for every subset $S$ of $\partial \Omega$ with positive distance from $L$, there exists $c(S)>0$ such that 
\[
v-A\geq c(S)\quad\text{ on } S\,.
\]
Combining the above estimates, there exists $\e_0>0$ depending on $r$ and $s$ such that for $\e<\e_0$,
the affine function
\[
A_\e(z):=A(z)+\e e_{12}^\bot\cdot (z-a_1) 
\]
satisfies  $A_\e=v_0$ on $[a_1a_2]$ and 
\[
A_\e\leq v_0 \quad\text{ on } \{a_1,\dots,a_N\}\cup\partial\Omega\,.
\]
By Lemma \ref{lem:v0LKv0eq}, $v_0=L_Kv_0$ with $K=\{a_1,\dots,a_N\}\cup\partial\Omega$, and hence the above implies that $A_\e$ is a supporting function for $v_0$, and in turn  that $p+\e e_{12}^\bot\in \partial^- v_0(x)$.
 \end{proof}

We now state the global structure of the solution to \eqref{eq: MAD}:

\begin{lemma}\label{lemma: 2d ridge}

Let $\Omega\subseteq \R^2$ be open, convex, bounded, with $C^{1,1}$ boundary.  Let $v_0\in C^0(\overline\Omega)$ be the solution to \eqref{eq: MAD}. 
Define the singular set 
\[
S_{v_0} := \{x\in \Omega\,:\, v_0\text{ is not differentiable at }x\}\,.
\]
Then
\begin{itemize}
\item [(i)]  There exists $\rid\subseteq  \{(i,j)\in \{1,\dots,N\}^2:i\neq j\}$ such that 
\[
S_{v_0} = \bigcup_{(i,j)\in \rid} [a_ia_j]\,.
\]

\item [(ii)] There exists $C>0$ only depending on $\Omega,\overline\mu$ such that
\[|\nabla^2 v_0(x)|\leq \frac{C}{\dist(x,\{a_1,\ldots,a_N\})}
\]
for almost every $x\in \Omega \setminus S_{v_0}$. In particular, $v_0\in W^{2,\infty}_\loc(\Omega \setminus S_{v_0})$.
\item [(iii)] For each $(i,j)\in\rid$, there exist $b_{ij}^+,b_{ij}^-\in\Omega$ satisfying  the following property:  $[a_ib_{ij}^+a_jb_{ij}^-]$ is a non-degenerate rhombus with diagonal $[a_ia_j]$, and $v_0$ is affine on $[a_ib_{ij}^+a_j]$ and on $[a_ib_{ij}^-a_j]$, with two different values of the gradient $\nabla v_0$ on these sets.

\item [(iv)] If $N\geq 2$, then for every $i\in \{1,...,N\}$ there exists $j\in \{1,\dots,N\}$ such that $(i,j)\in \rid$.
\end{itemize}
\end{lemma}

Note that if $\partial \Omega$ is not $C^{1,1}$, e.g. in the case of a square, there may be additional ridges connecting the vertices with $\partial \Omega$, and (ii) will fail.

\begin{remark}
  Whenever $(i,j)\in\rid$, we will say that $[a_ia_j]$ is a \emph{ridge}.
\end{remark}

\begin{proof}
Let $x_0\in \Omega\setminus\{a_1,\dots,a_N\}$ and $p\in \partial^-v_0(x_0)$. 
By Lemma \ref{lem: lifting}, the extreme points of ${(\partial^-v_0)^{-1}(p)}$ are contained in  $\partial \Omega \cup \{a_1,\ldots,a_N\}$. At least one of the following holds: (a) ${(\partial^-v_0)^{-1}(p)}\cap \partial\Omega\neq\emptyset$, or (b) $\#\partial^-v_0(x_0)>1$, or (c) ${(\partial^-v_0)^{-1}(p)}\cap \partial\Omega=\emptyset$ and $\#\partial^-v_0(x_0)=1$.

\begin{itemize}
\item [(a)]
    ${(\partial^-v_0)^{-1}(p)}\cap \partial\Omega\neq\emptyset$.
In this case, since $v_0(z) = 0 \geq v_0(x_0) + p\cdot(z-x_0)$ for all $z\in \partial\Omega$, $p$ must be a positive multiple of the outer normal $n(z_0)$ for any $z_0\in {(\partial^-v_0)^{-1}(p)}\cap \partial \Omega$. 
In addition, there must be at least one vertex $a\in (\partial^-v_0)^{-1}(p)\cap \{a_1,\ldots,a_N\}$, since otherwise $v_0=0$ on $(\partial^-v_0)^{-1}(p)$, which would imply  $v_0=0$ in $\overline\Omega$. This gives us the exact form
\begin{equation}\label{eq: subgradient fraction}
  p = \frac{v_0(a)}{n(z_0)\cdot(a-z_0)}n(z_0)\,.
\end{equation}
\item[(b)] $\#\partial^-v_0(x_0)>1$.
Since $0=\overline \mu(\{x_0\})=\Lm^2(\partial^- v_0(x_0))$, the subdifferential must be a line segment ${\partial^- v_0(x_0)} = [p^-p^+]$ with $p^+\neq p^-$. If $p\not\in\{p^-,p^+\}$, then
\[
  \begin{split}
    v_0(x_0)& + \tilde p\cdot (x-x_0)\leq\\
& \max(v_0(x_0) + p^-\cdot (x-x_0), v_0(x_0) + p^+\cdot (x-x_0))\quad\text{ for } x\in\Omega\,,
  \end{split}
\]
with equality holding for  $x\in x_0+\R(p^+-p^-)^\perp$. This shows that $(\partial^-v_0)^{-1}(p)$ is a line segment orthogonal to $p^+-p^-$ for such $p$.

By (a), an extreme point of such a  line segment cannot be a boundary point, so that $(\partial^-v_0)^{-1}(p) = [a_i a_j]$ for some $i,j\in \{1,\dots,N\}$. In addition, 
  \begin{equation}\label{eq:22}
    [a_ia_j]\subsetneq (\partial^-v_0)^{-1}(p^\pm)\,,
  \end{equation}
since the equality $[a_ia_j]= (\partial^-v_0)^{-1}(p^\pm)$ would imply the existence of  $\e_0>0$ such that 
\[
p^\pm+t e_{ij}^\bot\in \partial^-v_0(x_0)\quad\text{  for }|t|<\e_0
\]
by Lemma \ref{lem:v0aux}, where $e_{ij}=(a_j-a_i)/|a_j-a_i|$. This  would be a contradiction to the choice of $p^\pm$ as extreme points of $\partial^- v(x_0)$. Now \eqref{eq:22} implies that $\nabla v_0=p^-$, $\nabla v_0=p^+$ respectively on at least two non-degenerate triangles bordering  $[a_ia_j]$, which we may denote by $[a_ib_{ij}^+a_j]$, $[a_ib_{ij}^+a_j]$, with a choice of $b^\pm_{ij}$ that makes $[a_ib_{ij}^+a_jb^-_{ij}]$ a rhombus.

\item[(c)] $(\partial^-v_0)^{-1}(p)\cap \partial\Omega=\emptyset$ and $\#\partial^-v_0(x_0)=1$. 
In this case all extreme points of 
\[
{(\partial^-v_0)^{-1}(p)}={(\partial^-v_0)^{-1}(\nabla v_0(x_0))}
\]
are vertices, and we may write ${(\partial^-v_0)^{-1}(p)}=[a_{i_1}\dots a_{i_M}]$. Hence, if $x_0\in \partial {(\partial^-v_0)^{-1}(p)}$,  there exist $i,j\in\{1,\dots,N\}$ such that $x_0\in [a_ia_j]$ and 
\[
{(\partial^-v_0)^{-1}(p)}\cap \{x:(x-x_0)\cdot e_{ij}^\bot >0\}=0\,,
\]
which  produces the contradiction $\#\partial^-v_0(x_0)>1$ by Lemma \ref{lem:v0aux}.
Hence $x_0$ must be an interior point of $(\partial^-v_0)^{-1}(p)$, i.e. $v_0$ is affine in a neighborhood of $x_0$.
\end{itemize}
\emph{Proof of (i),(ii),(iii):}

In particular we may draw the following from the case distinction (a)-(c) above: If $x\in \Omega\setminus \cup_{(i,j)\in\rid}[a_ia_j]$, then either $v_0$ is affine in a neighborhood of $x$, or by \eqref{eq: subgradient fraction},
\begin{equation}
  \nabla v_0(x) =   p = \frac{v_0(a)}{n(z)\cdot(a-z)}n(z),
\end{equation}
for some vertex $a\in\{a_1,\ldots,a_N\}$ and some boundary point $z\in \partial \Omega$.  If $x\in [a z]$, the boundary point $z=z(x,a)$ is defined through the locally Lipschitz implicit function
\[
  \begin{cases}
    z(x,a) \in \partial \Omega\\
    x\in [a z(x,a)]\,.
  \end{cases}
\]

A direct calculation shows that $|\nabla_{x} z(x,a)|\leq \frac{C(\Omega,a)}{|x-a|}$. By the chain rule, $|\nabla^2 v_0(x)|\leq \frac{C}{|x-a|}$, where $C$ only depends on $\Omega$ and $\overline\mu$. This proves (i) and (ii), letting $\rid$ be the set of pairs $(i,j)$ that appear in  (b). Point (iii) follows immediately from the constructions in (b).

\emph{Proof of (iv):}
Assuming that $N>1$, consider a vertex $a\in\{a_1,\ldots,a_N\}$. We claim that there exists  $p\in \partial^- v_0(a)$ such that ${(\partial^-v_0)^{-1}(p)}$ contains at least one more vertex $a'\neq a$. Indeed, if this were not so, then by convexity of $\Omega$ we would have
\[
\Omega\subseteq\bigcup_{p\in \partial^-v_0(a)}(\partial^-v_0)^{-1}(p)\,,
\]
where for each of  the sets $(\partial^-v_0)^{-1}(p)$ on the right hand side, all extreme points except $a$ are contained in $\partial\Omega$. But  
 $p\in \partial^-v_0(a)$ and $x\in (\partial^-v_0)^{-1}(p)\setminus \{a\}$ implies  that $x$ is an element but not an extreme point of $(\partial^-v_0)^{-1}(p)$, which in turn implies $\mu_{v_0}(\{x\})=0$. Hence $\supp \overline\mu=\supp\mu_{v_0}=\{a\}$, which is  a contradiction to the assumption that the support of $\overline \mu=\mu_{v_0}$ contains at least two points. This proves the claim that there exists $a'\in\{a_1,\dots,a_N\}\setminus \{a|$ such that $[aa']\subseteq(\partial^-v_0)^{-1}(p)$.
We note  that $a'$ is an extreme point of  ${(\partial^-v_0)^{-1}(p)}$, since otherwise we would again obtain by  an analogous  argument to the one that we have just used, that $\mu_{v_0}(\{a'\})=0$. 

\medskip

So let $p\in \partial^-v_0(a)\cap \partial^-v_0(a')$ with $a,a'\in \{a_1,\dots,a_N\}$, $a\neq a'$, $a$ and $a'$ being extreme points of ${(\partial^-v_0)^{-1}(p)}$. Since $v_0=0$ on $\partial \Omega$ and $v_0<0$ on $\Omega$, we have that  ${(\partial^-v_0)^{-1}(p)}\cap \partial \Omega$ is either empty, consists of a single point, or of a line segment. 
 Since all extreme points of ${(\partial^-v_0)^{-1}(p)}$ lie either on $\partial \Omega$ or in $\{a_1,\ldots,a_N\}$, there must be a vertex $a''\in \partial(\partial^-v_0)^{-1}(p) \cap \{a_1,\ldots,a_N\} \setminus \{a\}$ (possibly, but not necessarily identical to $a'$) such that $[a\,a'']\subseteq \partial (\partial^-v_0)^{-1}(p)$. Without loss of generality, we may assume that the   conditions of Lemma \ref{lem:v0aux} are fulfilled for $a_1\equiv a$, $a_2\equiv a''$. Hence,  for $x\in [a a'']\setminus\{a,a''\}=[a_1 a_2]\setminus\{a_1,a_2\}$, $\#\partial^- v_0(x)>1$, which yields $(1,2)\in \rid$.
\end{proof}

\section{Proof of the upper bound}
\label{sec:proof-upper-bound}

\subsection{Proof of the upper bound for convex out-of-plane component}

\label{sec:ub-convex}

\newcommand{\w}{\mathbf{w}}

The following lemma is a statement about the distributional ``double curl'' of a $\R^{2\times 2}_{\mathrm{sym}}$ valued function $m$. We will later apply it to the case
\[m=\frac12\left(\nabla \bu+\nabla \bu^T+\nabla v\otimes\nabla v-\nabla v_0\otimes \nabla v_0\right)\,.\]
In this case, we note that the distribution  $\curl\curl m\in W^{-2,2}$  is given by 
\[
\curl\curl \frac12\left(\nabla \bu+\nabla \bu^T+\nabla v\otimes\nabla v-\nabla v_0\otimes \nabla v_0\right)=\mu_{v_0}-\mu_v\,\,,
\]
as can be seen by approximation of $v,v_0$ by smooth functions and a simple calculation.
This identity has already been used   to prove lower bounds in similar situations (see \cite{gladbach2024variational,olber2019crump,olbermann2018shape,Olbdisc}). 
In the statement of the lemma we will write $e(\bu):=\frac12(\nabla \bu+\nabla \bu^T)$.
\begin{lemma}
\label{lem:W-22L2equality}
Let $U\subset\R^2$ be open bounded and simply connected, and  $m\in L^2(U;\R^{2\times 2}_{\mathrm{sym}})$. Then
\[
\min_{\bu\in W^{1,2}(U)} \|e(\bu)+m\|_{L^2(U)}= \|\curl\curl m\|_{W^{-2,2}(U)}\,.
\]
\end{lemma}

\newcommand{\bv}{\mathbf{v}}

\begin{proof}
We start off by noting that
\[
\|\curl\curl m\|_{W^{-2,2}(U)}=\sup_{\varphi\in W^{2,2}_0(U)}\frac{\langle \curl\curl m,\varphi\rangle_{W^{-2,2},W^{2,2}_0}}{\|\nabla^2\varphi\|_{L^2(U)}}\,.
\]
For every $\varphi\in W^{2,2}_0(U)$ and $\bu\in W^{1,2}(U;\R^2)$, we obtain by two integrations by
parts, and the Cauchy-Schwarz inequality,
\[
\begin{split}
  \langle\curl\curl m,\varphi\rangle_{W^{-2,2},W^{2,2}_0}&=
\langle\curl\curl (m+e(\bu)),\varphi\rangle_{W^{-2,2},W^{2,2}_0}\\
&=  \int_{U} (m+e(\bu)):\cof \nabla ^2\varphi\d x\\
  &\leq  \|m+e(\bu)\|_{L^2(U)}
  \|\varphi\|_{W_0^{2,2}(U)}\,.
\end{split}
\]
This proves $\|\curl\curl m\|_{W^{-2,2}(U)}\leq \min_{\bu\in W^{1,2}(U)} \|e(\bu)+m\|_{L^2(U)}$. To show the opposite inequality, 
let $u$ be the minimizer of $\|e(\bu)+m\|_{L^2(U)}^2$.
Then $\div (e(\bu)+m)=0$ in the sense
\[
\int_U (e(\bu)+m)\cdot\nabla \bv\d x=0 \quad\text{ for all } \bv\in W^{1,2}(U;\R^2)\,, 
\]
which implies
\[
\int_{\R^2} (e(\bu)+m)\chi_U\cdot\nabla \bv\d x=0 \quad\text{ for all } \bv\in W^{1,2}(\R^2;\R^2)\,.
\]
Hence there exists $\psi\in W^{2,2}_{\loc}(\R^2)$ such that $\cof \nabla^2\psi=(e(\bu)+m)\chi_U$.  
By adding an affine function, we may assume $\psi=0$ on $\R^2\setminus U$. Thus $\psi|_U\in W^{2,2}_0(U)$ and 
\[
\|e(\bu)+m\|^2_{L^2(U)}=\int_U (e(\bu)+m )\cof \nabla^2 \psi\d x=\langle\curl\curl m, \psi\rangle_{W^{-2,2},W^{2,2}_0}
\]
which yields
\[
\|e(\bu)+m\|_{L^2(U)}=\frac{\langle \curl\curl m,\psi\rangle_{W^{-2,2},W^{2,2}_0}}{\|\nabla^2\psi\|_{L^2(U)}}\,.
\]
\end{proof}

\begin{proof}[Proof of the upper bound in Theorem \ref{thm:main}]
For $(i,j)\in\rid$ (see Lemma \ref{lemma: 2d ridge} for the notation), we will write
\[
e_{ij} := \frac{a_j-a_i}{|a_j-a_i|}\,.
\]
According to Lemma \ref{lemma: 2d ridge}, for each $(i,j)\in\rid$, we may choose two points $b_{ij}^+, b_{ij}^-$ such that $\nabla v_0=p^+$ on $[a_ib_{ij}^+a_j]$ and $\nabla v_0=p^-$ on $[a_ib_{ij}^-a_j]$, for some $p^+,p^-\in\R^2$ with  $p^+\neq p^-$. 
 We may define the convex one-homogeneous function $W_{ij}:\R\to\R$ by requiring
\[
v_0(x)=v_0(a_i+e_{ij}\cdot (x-a_i) e_{ij})+ W_{ij}((x-a_i)\cdot e_{ij}^\bot)\quad \text{ for }x\in R_{ij}=[a_ib_{ij}^+a_jb_{ij}^-]\,.
\]
Now choose a convex $C^2$ function $w_{ij}:\R\to\R$ and $M_{ij}>0$ such that $w_{ij}\geq W_{ij}$, $w_{ij}(x)=W_{ij}(x)$ for $|x|>M_{ij}$ and 
\[
\int_{-M_{ij}}^{M_{ij}} (W'_{ij}(t))^2-(w_{ij}'(t))^2\d t=W_{ij}'(M_{ij})-W_{ij}'(-M_{ij})\,.
\]
Such a choice of $w_{ij},M_{ij}$ is always possible (note that the right hand side above does not depend on $M_{ij}$) and depends only on the two values of $W_{ij}'$ (and hence on $v_0)$.  
For $(i,j)\in \rid$ let $S_{ij}^\e$ denote the strip of width $2\e M_{ij}$ around $[a_i a_j]$, 
\[
S_{ij}^\e := \{x:(x-a_i)\cdot e_{ij} \in [0,|a_i-a_j|], \e^{-1}(x-a_i)\cdot e_{ij}^\bot \in[-M_{ij},M_{ij}]\}\,.
\]

For $\e>0$ we may  define the convex function $w_{ij,\e}:\Omega\to \R$, 
\[
w_{ij,\e}(x) := v_0(a_i+e_{ij}\cdot (x-a_i) e_{ij})+ \e w_{ij}(\e^{-1}(x-a_i)\cdot e_{ij}^\bot)
\]
We note that there exists $C_1>0$ that only depends on $v_0$ (and hence on $\Omega,\overline\mu$) such that
\[
v_0(x)\geq w_{ij,\e}(x) \text{ for } x\in \Omega \setminus \left(S_{ij}^\e\cup B(a_i,C_1\e)\cup B(a_j,C_1\e)\right)\text{ for all } (i,j)\in\rid\,,
\]
and 
\[
\|\nabla^2 w_{ij,\e}\|_{L^\infty}\lesssim \e^{-1}\,.
\]
Now we define the convex function 
\[
\tilde v_\e(x) := \max\{ v_0(x), w_{ij,\e}(x):(i,j)\in\rid\}\quad\text{ for } x\in\Omega\,.
\]
Writing
\[
  \begin{split}
B_\e&=\bigcup_{i=1}^N B(a_i,C_1\e)\\
S_\e&=\bigcup_{(i,j)\in\rid}^N S_{ij}^\e\setminus B_\e\,,
\end{split}
\]
our construction guarantees that
\[
\tilde v_\e\in C^2\left(\Omega\setminus B_\e \right)\,,
\]
with
  \begin{equation}
\label{eq:6}
\|\nabla^2 \tilde v_\e\|_{L^\infty\left(\Omega\setminus B_\e\right)}\leq C\e^{-1}\,\,,
\end{equation}
and 
\[
  \begin{split}
\|\nabla^2 \tilde v_\e\|_{L^2\left(\Omega\setminus B_\e\right)}^2&\lesssim 
\int_{\Omega\setminus(B_\e \cup S_\e)}|\nabla^2 v_0|^2\d x+\sum_{(i,j)\in \rid}\int_{S_{ij}^\e}|\nabla^2 w_{ij,\e}|^2\d x\\
&\leq C \left(\frac{1}{\log \e}+\e^{-1}\right)\,.
\end{split}
\]
We further modify $\tilde v_\e$ in order to obtain an analogous estimate on all of $\Omega$ by using Lemma \ref{lem:convex_smoothing}. More precisely, we extend $\tilde v_\e$ to a globally Lipschitz convex function on $\R^2$, and choose some convex open set $\tilde\Omega\cc\Omega$ with $\{a_1,\dots,a_N\}\subseteq\tilde\Omega$. We choose  $c>0$ and $\e_0$ that only depend on $v_0,C_1,\tilde\Omega$  such that for every $\e\in (0,\e_0)$ the following are fulfilled:  
\begin{itemize}
\item 
    The function $\hat v_\e:=P_{c\e}\tilde v_\e$
   satisfies $\hat v_\e=\tilde v_\e$ on $\tilde \Omega\setminus B_\e$,
\item $\|\nabla^2 v_0\|_{L^\infty(\Omega\setminus\tilde \Omega)}\leq C\e^{-1}$.
\end{itemize}
Then we set
\[
v_\e(x) :=
\begin{cases}
  v_0(x) &\text{ if } x\in \Omega\setminus\tilde\Omega\\
\hat v_\e(x)&\text{ if } x\in \tilde\Omega\,.
\end{cases}
\]
By construction this choice satisfies 
\[
  \begin{split}
    \|\nabla^2 v_\e\|_{L^\infty(\Omega)}&\leq C\e^{-1}\\
v_\e&=w_{ij,\e} \quad\text{ in  }S_{ij}^{\e}\setminus B_\e\\
    v_\e&=0\quad\text{ on }\partial\Omega\,.
  \end{split}
\]

Next we define the in-plane deformation $\bu_\e$. 
We first note that 
\[
\curl\curl\left( \nabla v_\e\otimes\nabla v_\e-\nabla v_0\otimes \nabla v_0\right)
\]
has support in 
\[
\cup_{i=1}^N B(a_i,C_1\e)\,.
\]
For $i=1,\dots,N$ let $\mathcal V^i_0,\mathcal V_\e^i$ be defined by 
\[
  \begin{split}
    \mathcal V^i_0(x)&:=\sup \{p\cdot x+v_0(x):p\in \partial^-v_0(B(a_i,C_1\e))\}\\
    \mathcal V^i_\e(x)&:=\sup \{p\cdot x+v_\e(x):p\in \partial^-v_\e(B(a_i,C_1\e))\}
  \end{split}
\]
With this definition we have 
\[
  \begin{split}
    \nabla v_\e\otimes\nabla v_\e-\nabla v_0\otimes\nabla v_0-\left(\nabla \mathcal V^i_\e\otimes\nabla \mathcal V^i_\e-\nabla\mathcal V^i_0\otimes\nabla \mathcal V^i_0\right)&=0 \quad\text{ on } B(a_i,\e)\\
    \curl\curl\left(\nabla \mathcal V^i_\e\otimes\nabla \mathcal V^i_\e-\nabla\mathcal V^i_0\otimes\nabla \mathcal V^i_0\right)&=0\quad\text{ on }\Omega\setminus B(a_i,C_1\e)\,.
  \end{split}
\]
Let us write
\[
m_i := \nabla \mathcal V^i_\e\otimes\nabla \mathcal V^i_\e-\nabla\mathcal V^i_0\otimes\nabla \mathcal V^i_0 \,.
\]
By Lemma \ref{lem:W-22L2equality}, 
\[
\|\curl\curl m_i\|_{W^{-2,2}(\Omega)}\leq \|m_i\|_{L^2(\Omega)}\leq C\e^2\,.
\]
Furthermore, 
\[
  \begin{split}
    \min_{\bu\in W^{1,2}(\Omega)} \left\|e(\bu)+\sum_{i=1}^N m_i\right\|_{L^2}&\leq \left\|\curl\curl\sum_{i=1}^N m_i\right\|_{W^{-2,2}(\Omega)}\\
                                             &\leq \sum_{i=1}^N \|\curl\curl m_i\|_{W^{-2,2}(\Omega)}\\
                                             &\leq C\e^2\,.
  \end{split}
\]
Choosing $\bu_\e$ as the minimizer of the left hand side above, and  $\e:=h^{2/3}$, we obtain 
\[
E_h(\bu_\e,v_\e)\leq C (\e^2+h^2\e^{-1})\leq C h^{4/3}\,.
\]
 \end{proof}


\subsection{Upper bound for non-convex out-of-plane component}
\label{sec:ub-non-convex}
If we allow for non-convex out-of-plane deformations $v$, then we may achieve the scaling $h^{5/3}$. In the statement below, $\Omega\subseteq\R^2$ and $\overline\mu=\sum_{i=1}^N\sigma_i\delta_{a_i}$  are as in   Theorem \ref{thm:main}. 

\begin{proposition}
\label{prop:non-convex}
  There exists $C>0$ that only depends on $\Omega, \overline\mu$ with the following property: For $h>0$ there exists $(\bu_h,v_h)\in W^{1,2}(\Omega;\R^2)\times W^{2,\infty}(\Omega)$ such that $v_h = 0$ on $\partial\Omega$ and
\[
E_h(\bu_h,v_h)\leq C h^{5/3}\,.
\]
\end{proposition}

Before proving Proposition \ref{prop:non-convex}, we note that in combination with Theorem \ref{thm:main} it implies  that minimizers of $E_h$ have non-convex out-of-plane component:

\begin{corollary}
\label{cor:min-non-convex}
There exists $h_1<h_0$ such that the following holds true for every $0<h<h_1$: If $(\bu_h,v_h)\in \argmin_{v=0\text{ on }\partial\Omega} E_h$, then $v_h $ is non-convex.
\end{corollary}

The proof of Proposition \ref{prop:non-convex} is based on  Lemma 16 from \cite{conti2015symmetry}, which itself is an adaptation of the construction for the single ridge by Conti and Maggi \cite{MR2358334} to the F\"oppl--von K\'arm\'an approximation. Below we  cite the F\"oppl--von K\'arm\'an version in full for the reader's convenience. In the lemma we use the notation 
\[
\overline E_h[\bu, v, W]=\int_W |\nabla\bu+\nabla\bu^T+\nabla v\otimes\nabla v|^2+h^2|\nabla ^2 v|^2 \d x\,.
\]
\begin{lemma}{\cite[Lemma 16]{conti2015symmetry}}
\label{lem:ridge}
 Let $[abcd]\subseteq \R^2$ be a non-degenerate quadrilateral with diagonals $[ac]$
 and $[bd]$, contained in the square with diagonal $[ac]$. Furthermore, let
 $(\bu_0,w_0):[abcd]\to \R^2\times \R$ be continuous functions,
 affine on $[abc]$ and $[adc]$, with $\nabla \bu_0+\nabla \bu_0^T+\nabla w_0\otimes \nabla w_0=0$. Then for all $h\in (0,l/8)$ there are $(\bu_h, w_h)\in C^2([abcd]\setminus\{a,c\};\R^2\times \R)$ 
 such that $\bu_0=\bu_h$, $\nabla \bu_0=\nabla \bu_h$, $w_0=w_h$ on $\partial[abcd]\setminus\{a,c\}$, $|\nabla \bu_h|+|\nabla w_h|\le C(\bu_0,w_0)$ and
 \begin{equation}\label{eq:11}
  \overline E_h[\bu_h,w_h, [abcd]\setminus (B(a,h) \cup B(c,h))]\le C h^{5/3}l^{1/3}\,,
 \end{equation}
where $l=|a-c|$, and
the constant $C$ may depend on the angles of $[abcd]$, and on $\bu_0$ and $w_0$, but not on $h$.
\end{lemma}

\begin{proof}[Proof of Proposition \ref{prop:non-convex}]
Using the notations of Lemmata \ref{lemma: 2d ridge} and \ref{lem:ridge}, for every pair $(i,j)\in\rid$, we may write $R_{ij}=[abcd]$ with $a_i=a$, $a_j=c$, set $w_0^{ij}=v_0$ and define $\bu_0^{ij}$ by requiring it to be continuous on $[abcd]$ and affine on both $[abc]$ and $[adc]$, with
  \begin{equation}\label{eq:10}
\nabla \bu_0^{ij}+ \left(\nabla\bu_0^{ij}\right)^T=-\nabla v_0\otimes \nabla  v_0\,,
\end{equation}
where $v_0$ is the unique solution of \eqref{eq: MAD}.
The equation \eqref{eq:10} indeed possesses a solution: Note that after a suitable rotation of the domain, we may assume that $[ac]\subseteq\R\times\{0\}$, and 
\[
 v_0(x)=\begin{cases} A_1x_1+A_2x_2&\text{ if }x_2\geq 0\\
    A_1 x_1+A_3x_2&\text{ if }x_2< 0\,.\end{cases}
\]
Then a solution of \eqref{eq:10} is given by 
\[
  \bu_0^{ij}(x_1,x_2)=\begin{cases}(A_4x_1+A_5x_2,A_6x_2) &\text{ if }x_2\geq 0\\
    (A_4 x_1+A_7 x_2,A_8x_2)&\text{ if }x_2< 0\end{cases}
\]
where $A_4,\dots,A_8$ are given by 
\[
-2A_4=A_1^2, \quad -A_5=A_1A_2,\quad -2A_6=A_2^2,\quad  -A_7=A_1A_3, \quad-2A_8=A_3^2\,.
\]
Applying Lemma \ref{lem:ridge} we obtain $(\bu_h^{ij},w_h^{ij})$ in $C^2([abcd]\setminus\{a,c\};\R^2\times\R)$ that satisfy the bound \eqref{eq:11}, and $\bu_h^{ij}=\bu_0^{ij}$, $w_h^{ij}=v_0$ on $\partial[abcd]$. This finishes the construction on $R_{ij}$. Now we set
\[
\bu_h=
\begin{cases}
\bu_h^{ij}-\bu_0^{ij}& \text{ in }R_{ij}  \text{ for }(i,j)\in \rid\\
0 & \text{ in }\Omega\setminus \bigcup_{(i,j)\in \rid}R_{ij}
\end{cases}
\]
and 
\[
v_h=
\begin{cases}
  w_h^{ij}&\text{ in }R_{ij}  \text{ for }(i,j)\in \rid\\
v_0 & \text{ in }\Omega\setminus \bigcup_{(i,j)\in \rid}R_{ij}\,.
\end{cases}
\]
The energy satisfies
\begin{equation*}
  \begin{split}
\int_{\Omega\setminus  \bigcup_{i=1}^N B(a_i,h)}&|\nabla \bu_h+\nabla \bu_h^T+\nabla v_h\otimes\nabla v_h-\nabla v_0\otimes \nabla v_0|^2+h^2|\nabla^2v_h|^2\d x\\
&=\sum_{(i,j)\in\rid}\overline E_h\left[\bu_h^{ij},w_h^{ij},R_{ij}\setminus (B(a_i,h)\cup B(a_j,h))\right]\\
&\quad +h^2\int_{\Omega\setminus(\cup_{(i,j)\in \rid}R_{ij}\cup(\cup_{i=1}^N B(a_i,h)))}|\nabla^2v_0|^2\d x\\
& \le C h^{5/3}\,.
\end{split}
\end{equation*}
It remains  to smoothen our construction on $\bigcup_{i=1}^N B(a_i,h)$. 
For notational simplicity we discuss only the treatment of one vertex placed in the origin, $a_1=(0,0)$.
Let $\phi_h\in C^\infty_c(B(0,h))$ be such that $\phi=1$ on $B(0,h/2)$. Let
$\eta\in C^\infty_c(B(0,1))$ be a standard mollifier, and let $\eta_h(\cdot)=h^{-2}\eta(\cdot/h)$.
We set $\tilde v_h=\phi_h (v_h\ast \eta_h) + (1-\phi_h)v_h$.
Then $\|\nabla^2\tilde v_h\|^2_{L^2(B(0,h))} \le C$. We note that $\nabla v_h$ is uniformly bounded, therefore 
$\|\nabla v_h\otimes \nabla v_h-\nabla \tilde v_h\otimes \nabla \tilde v_h\|^2_{L^2(B(0,h))} \le C h^2$. We conclude that
\begin{equation*}
  \begin{split}
 E_h(\bu_h, \tilde v_h)
&=\int_{\Omega}|\nabla \bu_h+\nabla \bu_h^T+\nabla \tilde v_h\otimes\nabla \tilde v_h-\nabla v_0\otimes \nabla v_0|^2+h^2|\nabla^2v_h|^2\d x\\
& \le C h^{5/3}.
\end{split}
\end{equation*} 
This proves the proposition.
\end{proof}

\section{Proof of the lower bound}
\label{sec:lb}
In the present section, we will prove the lower bound in Theorem \ref{thm:main}. This will proceed in several steps. As before, $v_0$ denotes the unique solution of the Monge-Amp\`ere equation \eqref{eq: MAD}.

\medskip


We will  consider sequences $v_h$ that satisfy the upper bound in Theorem \ref{thm:main}. By Lemma \ref{lem:W-22L2equality}, this implies $I_h(v_h)\lesssim h^{4/3}$. For such a sequence, we will  first prove ``suboptimal'' estimates for $\|v_h-v_0\|_{L^\infty}$. More precisely, we will prove 
in Subsection \ref{sec:suboptimal}  that $\|v_h-v_0\|_{L^\infty(\Omega)}\lesssim h^\gamma$ for some fixed positive $\gamma$. We call these estimates suboptimal since we will prove much sharper ones later on. 

In Subsection \ref{sec:estim-nearly-conic}, we consider the behavior of $v_h$ satisfying the upper bound near any one of the vertices $a_i$, $i=1,\dots,N$. By the suboptimal estimates obtained in Subsection \ref{sec:suboptimal}, we know that for a suitably chosen affine function $\tilde \ell^i$,  $v_h$ is close to the ``conical'' function $\hat v:=L_{\{v_h=\tilde \ell^i\}\cup\{a_i\}}v$ on $\{v_h\leq \tilde \ell^i\}$, see Figures \ref{fig:singlevertexcut} and \ref{fig:sketch2}. We prove a statement that allows us to upgrade the suboptimal estimates to \emph{optimal} estimates for $|v_h(x)-\hat v(x)|$ on $\{v_h\leq \tilde \ell^i\}$, see Proposition \ref{prop:conical_estimate}. Here and in the following, by ``optimal'' estimates we mean that one can construct examples in which the estimate is attained up to  at most a logarithmic factor  in the small parameter $h$.

\begin{figure}[h]
  \centering
  \includegraphics[height=5cm]{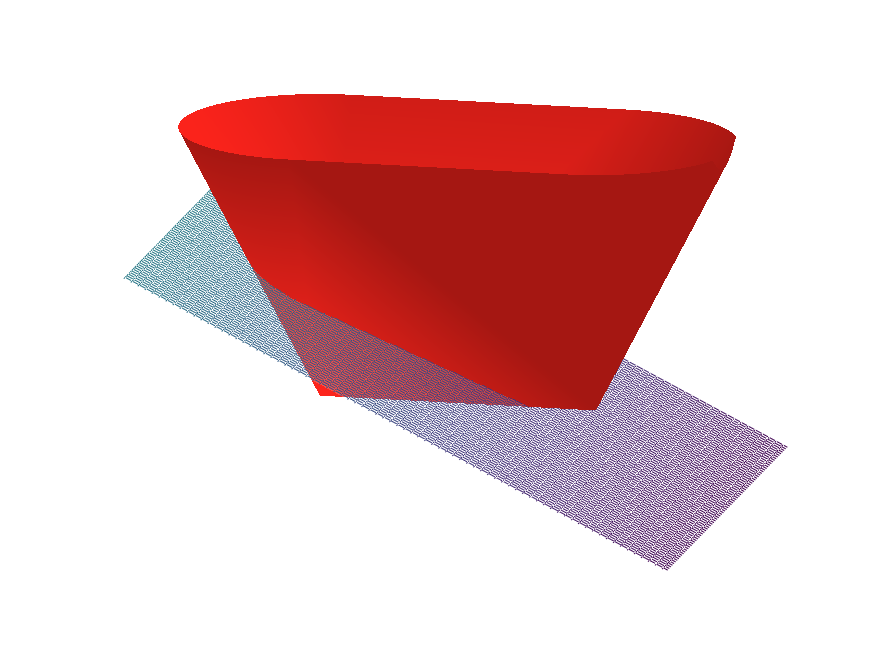}
\caption{The graphs of $v_h$ and $\tilde \ell^i$, with $v_h(a_i)<\tilde\ell^i(a_i)$, $v_h(a_l)>\tilde \ell^k(a_l)$ for $l\in\{1,\dots,N\}\setminus\{i\}$. The map $\tilde \ell^i$ will be constructed in Proposition \ref{prop: ridge} as a slight perturbation of some other affine map $ \ell^i$, introduced in Lemma \ref{lem:ridgeaux}.\label{fig:singlevertexcut}} \end{figure}


In Subsection  \ref{sec:ridges}, we consider a ridge $(i,j)\in \rid$ (where we are using the notation introduced in Lemma \ref{lemma: 2d ridge}), and the rhombus $R_{ij}=[a_ib_{ij}^+a_jb_{ij}^-]$, on which the gradient of $v_0$ has exactly two values. We use the results from the previous subsections as follows: By the suboptimal estimates from Subsection \ref{sec:suboptimal}, $v$ is close to $v_0$ on the corners of $R_{ij}$. Consider the continuous function $w_h$ that is affine on $[a_ib^+_{ij}a_j]$ and on $[a_ib^-_{ij}a_j]$, with the values of $w_h$ at  $a_i,a_j,b_{ij}^\pm$  given by those of $v_h$. By the results of Subsection \ref{sec:estim-nearly-conic}, we have that $v_h$ is close to $w_h$ on $\partial R_{ij}$, with \emph{optimal} estimates, see Figures \ref{fig:sketch2} and \ref{fig:sketch3}. Using the generalized monotonicity property of the Monge-Amp\`ere measure (Proposition \ref{prop: comparison}), we are able to transfer these estimates from the boundary to the bulk of $R_{ij}$. This crucial step is contained in Proposition \ref{prop: ridge}.

\begin{figure}
  \centering
  \includegraphics[height=5cm]{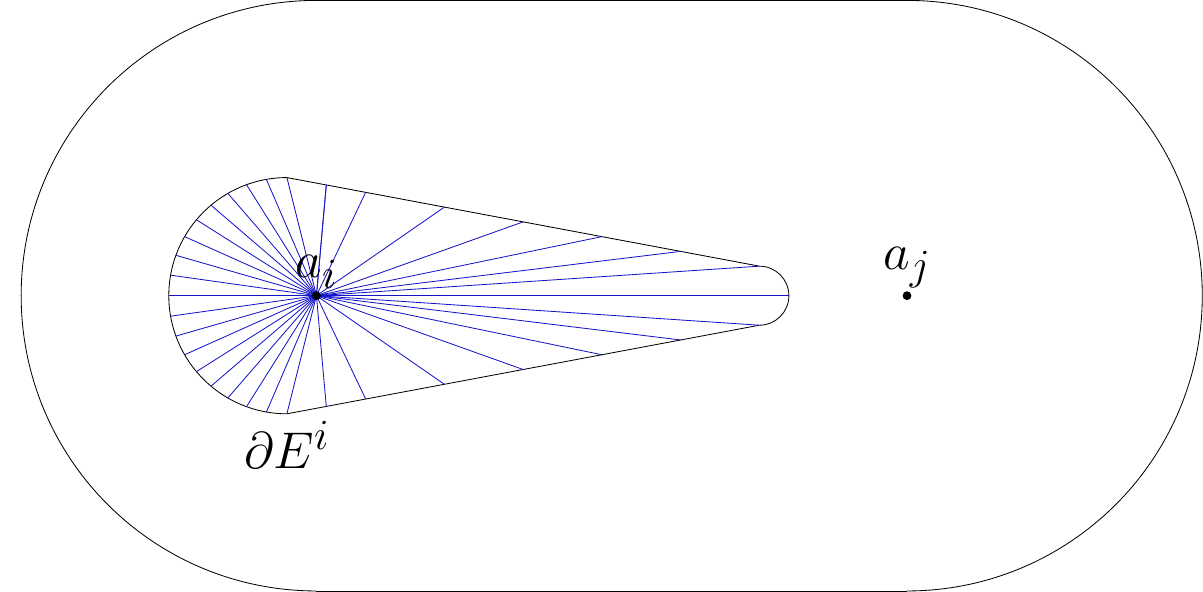}
\caption{The set $E^i=\{x:v_h(x)\leq \tilde \ell^i(x)\}$. The function $\hat v$ is defined by requiring to be affine on each of the line segments $[a_ix]$ with $x\in \partial E^i$, agreeing with $v_h$ on the endpoints of these line segments.  By Proposition \ref{prop:conical_estimate}, $\|v_h-\hat v\|_{L^\infty}$ is  controlled by the energy. This  estimate is optimal in the sense explained in the outline of the proof at the beginning of Section \ref{sec:lb}.  \label{fig:sketch2}}
\end{figure}

\begin{figure}
  \centering
\includegraphics[height=5cm]{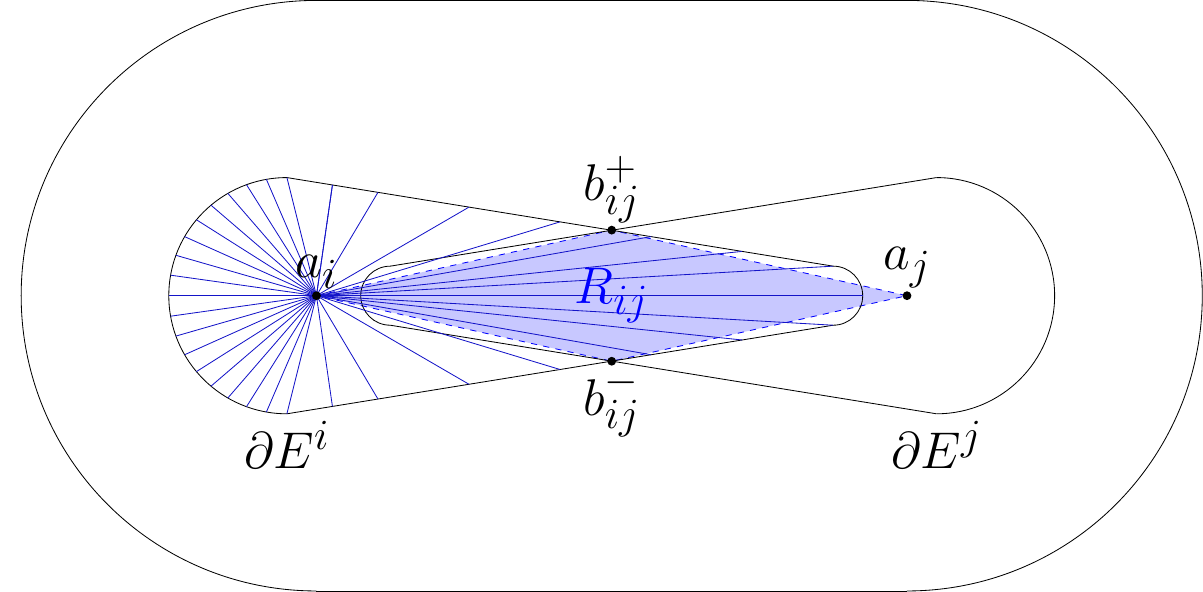}
  \caption{Since $ b^\pm_{ij}\in\partial E^i\cap \partial E^j$ and $R_{ij}=[a_i b_{ij}^+a_j b_{ij}^-]$, the optimal estimates for $v_h$ hold on all of $\partial R_{ij}$. These estimates control the distance of $v_h$ from a continuous function on $\partial R_{ij}$ that is affine on $[a_i b_{ij}^+a_j]$ and on $[a_i b_{ij}^-a_j]$. \label{fig:sketch3}}
\end{figure}

The latter allows us to prove the lower bound in Subsection  \ref{sec:lower_bound}: The optimal bounds in the interior imply that there is a small transition layer parallel to the ridge of size $\approx h^{-2/3}$ (up to logarithms in $h$), across which the gradient $\nabla v_h$ has a jump of order one. This is enough to establish the lower bound.

\begin{remark}
\label{rem:log_lb}
  The logarithmic factors in the lower bound  have the following source: When estimating the $W^{-2,2}$ norm of $\mu_v-\overline\mu$ from below (in the ``optimal'' estimates Propositions \ref{prop:conical_estimate} and \ref{prop: ridge}), we have to work with test functions that behave approximately like $x\mapsto|x-a_i|$, which fail to be  elements of $W^{2,2}$, producing  logarithmic singularities after smoothing on the appropriate scale. It seems that this cannot be improved within our method of proof, since we need to work with non-negative test functions that vanish at vertices (where $\overline\mu$ is concentrated) and are positive everywhere else in order to detect the deviations of $v$ from $v_0$. 
\end{remark}

\subsection{Suboptimal global estimates}
\label{sec:suboptimal}

We begin by stating that the solution of the discrete Monge-Amp\`ere Dirichlet problem \eqref{eq: MAD} is stable under a change in the weights of $\mu$:

\begin{lemma}\label{lem:stability}
  Let $\Omega\subseteq \R^n$ be open, bounded, convex. Let $\mu = \sum_{i=1}^N \sigma_i \delta_{a_i}, \nu = \sum_{i=1}^N \tau_i \delta_{a_i}\in \mathcal{M}_+(\Omega)$ be two non-negative discrete measures with the same support. Let $u,v\in C^0(\overline \Omega)$ be the respective solutions to the Dirichlet problems
  \begin{equation}
    \begin{cases}
      \mu_u = \mu , \mu_v = \nu&\text{ in }\Omega\\
      u=v=0 &\text{ on }\partial \Omega.
    \end{cases}
  \end{equation} 

Then
\begin{equation}
\|u-v\|_{L^\infty} \leq C(\Omega, N) \sum_{i=1}^N |\sigma_i^{1/n} - \tau_i^{1/n}|.
\end{equation}
\end{lemma}

\begin{proof}
Let $\delta = \|u-v\|_{L^\infty} >0$. Without loss of generality, there is $x\in \Omega$ with $\delta = u(x) - v(x)$. Let $A:=(\partial^-v)^{-1}(\partial^-v(x))\subseteq \overline \Omega$ be the maximal simplex containing $x$ where $v$ is affine. Since $u-v$ is convex on $A$, its maximum is attained on the extreme points  of $A$. By Lemmata \ref{lem: lifting} and \ref{lem:v0LKv0eq}, the extreme points of $A$ are contained in $\partial\Omega\cup\{a_1,\dots,a_N\}$. Hence the maximum must be attained at some vertex, so that without loss of generality $x = a_1$. For $k =  0,\ldots,N$ let $A_k = \{a_j\,:\,u(a_j) - v(a_j) \geq \frac{k\delta}{N}\}$. Since $A_k \supseteq A_{k+1}$ and $A_N \supseteq \{a_1\}$, there is $k_0\in \{0,\ldots,N-1\}$ with $A_{k_0} = A_{k_0+1}$. We will show that 
\begin{equation}\label{eq: subgradients}
\partial^- v(A_{k_0}) \supseteq B\left(\partial^- u(A_{k_0}), \frac{\delta}{N\diam(\Omega)}\right),
\end{equation}
 so that by the Brunn-Minkowski inequality
\[
  \mu_v(A_{k_0})^{1/n} \geq \mu_u(A_{k_0})^{1/n} + c(n)\frac{\delta}{N\diam(\Omega)},
\]
which after rearranging yields
\[
  \|u-v\|_{L^\infty} = \delta \leq C(\Omega,N)(\mu_v(A_{k_0})^{1/n} - \mu_u(A_{k_0})^{1/n}) \leq C(\Omega,N) \sum_{i=1}^N |\sigma_i^{1/n} - \tau_i^{1/n}|.
\]

To show \eqref{eq: subgradients}, let $p\in \partial^- u(a_j)$ for some $a_j\in A_{k_0} = A_{k_0+1}$ and $|q|< \frac{\delta}{N\diam(\Omega)}$. We estimate
\[
  \begin{aligned}
(p+q)\cdot a_j - v(a_j) \geq & p\cdot a_j - u(a_j) + \frac{(k_0+1)\delta}{N} + q\cdot a_j  \\
\geq & \max_{y\in \partial\Omega \cup \{a_1,\ldots,a_N\} \setminus A_{k_0}} p\cdot y - u(y) + \frac{(k_0+1)\delta}{N} + q\cdot y - |q||y-a_j|\\
\geq & \max_{y\in \partial\Omega \cup \{a_1,\ldots,a_N\} \setminus A_{k_0}} (p+q)\cdot y - v(y) + \frac{\delta}{N}  - \frac{\delta}{N}.
  \end{aligned}
\]
By Lemma \ref{lem: lifting}, $\max_{y\in \partial\Omega \cup \{a_1,\ldots,a_N\} } (p+q)\cdot y - v(y)=\max_{y\in\Omega}(p+q)\cdot y - v(y)$, and hence  $p+q\in \partial^- v(A_{k_0})$, completing the proof.
\end{proof}


In the upcoming lemma, we will use the notation 
\[
U_\rho := \{x\in U:\,\dist(x,\partial U)>\rho\}\,.
\]
for $U\subseteq\R^2$. 
The assumptions that will be made on $U$ are chosen such that they accommodate the case $U=\Omega\setminus\{a_1,\dots,a_N\}$, where $\Omega$ and $\{a_1,\dots,a_N\}$ are as in Theorem \ref{thm:main}.

\begin{lemma}
\label{lem:c0alphaestim}
Let $U \subseteq \R^2$ be open, bounded,  such that $\overline U$ is convex and $U_\rho$ has Lipschitz boundary, with 
\[
\L^2\left(\left\{x:\dist(x,\partial U_{\rho/2})<\frac{\rho}{3}\right\}\right)<C(U) \rho
\]
for all $\rho>0$. Furthermore let
 $\alpha \in (0,1]$, and $u\in C^{0,\alpha}(\overline U)$ convex. Then 
\begin{equation}\label{eq: u-Lu}
\|u-L_{\partial U}u\|_{L^\infty(U)} \leq C(U,\alpha)\|\mu_u\|_{W^{-2,2}(U)}^{\frac{4\alpha }{2(4\alpha  + 3)}}\|u\|_{C^{0,\alpha}}^{\frac{3}{4\alpha+3}}\,.
\end{equation}
\end{lemma}

\begin{proof}
Let $\rho>0$ to be chosen later.  Let $\phi_\rho=\eta_{\rho/3}*\mathds{1}_{U_{\rho/2}}$, where $\eta_{\e}=\e^{-2}\eta(\cdot/\e)$ and $\eta\in C^\infty_c(\R^n)$ is a standard symmetric mollifier with $\supp\eta\subseteq B(0,1)$. Then $\phi_\rho\in C_c^\infty(U)$ with $\phi_\rho = 1$ in $U_\rho$ and 
\[
  \begin{split}
    \|\phi_\rho\|_{W^{2,2}_0}&\leq \left(\int_{U}|\nabla^2\phi_\rho|^2\d x\right)^{1/2}\\
                           &= \left(\int_{\{x:\dist(x,\partial U_{\rho/2})<\frac{\rho}{3}\}}|\nabla^2\phi_\rho|^2\d x\right)^{1/2}\\
                           &\leq \left(C\rho^{-4}\rho\right)^{1/2}\\
                           &\leq C\rho^{-3/2}\,.
  \end{split}
\]
Furthermore, 
\[
\mu_u(U_\rho)  \leq \int\phi_\rho \d\mu_u \leq \|\mu_u\|_{W^{-2,2}(U)} \|\phi_\rho\|_{W^{2,2}_0(U)} \leq C(U)\|\mu_u\|_{W^{-2,2}(U)}\rho^{-3/2}\,. 
\]
By the Alexandrov maximum principle, Lemma \ref{lemma: Alexandrov}, we have 
\[
\|u-L_{\partial U_\rho}u\|_{L^\infty(U_\rho)} \leq C(U)(\mu_u(U_\rho))^{1/2} \leq C(U)\|\mu_u\|_{W^{-2,2}}^{1/2}\rho^{-3/4}\,.
\]

On the other hand, we note that  $L_{\partial U\cup \partial U_\rho} u=L_{\partial U_\rho}u$ on $U_\rho$, and
\[
\|L_{\partial U} u - L_{\partial U\cup \partial U_\rho} u\|_{L^\infty(U)} = \|L_{\partial U} u - u\|_{L^\infty(\partial U_\rho)} \leq C(U,\alpha) \rho^\alpha \|u\|_{C^{0,\alpha}}.   
\]
Similarly, 
\[
\|u-L_{\partial U\cup \partial U_\rho} u\|_{L^\infty(U\setminus U_\rho)}\leq C(U,\alpha) \rho^\alpha \|u\|_{C^{0,\alpha}}\,.
\]
Putting everything together via the triangle inequality 
\[
  \begin{split}
\|u-L_{\partial U} u\|_{L^\infty(U)}&\leq \|u-L_{\partial U_\rho}u\|_{L^\infty(U_\rho)} +\|u-L_{\partial U\cup \partial U_\rho} u\|_{L^\infty(U\setminus U_\rho)}\\
&\quad +\|L_{\partial U} u - L_{\partial U\cup \partial U_\rho} u\|_{L^\infty(U)}
\end{split}
\]
and choosing the optimal $\rho$,
\[
  \rho =\left( \frac{\|\mu_u\|_{W^{-2,2}}^{1/2}}{\|u\|_{C^{0,\alpha}}} \right)^{\frac{1}{3/4+\alpha}}\,,
\]
yields \eqref{eq: u-Lu}.
\end{proof}

From now on we assume that  $v_h\in W^{2,2}(\Omega)$ satisfies the upper bound in Theorem \ref{thm:main}, in particular (using Lemma \ref{lem:W-22L2equality})
  \begin{equation}\label{eq:27}
I_h(v_h)\leq C_1 h^{4/3}\,.
\end{equation}

\begin{lemma}
\label{lem:lifting_comparison}
  Let $\overline \mu = \sum_{i=1}^N \sigma_i \delta_{a_i}\in \M_+(\Omega)$ and $\alpha\in(\frac34,1)$. Then there exist $C,h_0>0$ depending on $\Omega, C_1,\alpha,a_1,\dots,a_N$ such that for all $h<h_0$, and all
 $v_h\in W^{2,2}(\Omega)$  with 
$v_h = 0$ on $\partial \Omega$ satisfying \eqref{eq:27},
we have that 
\[
\|v_h-L_{\partial \Omega\cup\{a_1,\dots,a_N\}}v_h\|_{L^\infty}\leq C h^{\gamma}\,,
\]
where 
\[
\gamma\equiv \gamma(\alpha)=\frac{4\alpha-3}{3(4\alpha+3)}>0\,.
\]
\end{lemma}

\begin{proof}
By the continuous embedding $W^{2,2}(\Omega)\hookrightarrow C^{0,\alpha}(\Omega)$, we have that $\|v_h\|_{C^{0,\alpha}(\Omega)}\leq C h^{-1/3}$. Thus we may apply Lemma \ref{lem:c0alphaestim} with $U=\Omega\setminus\{a_1,\dots,a_N\}$  to obtain 
\[
  \begin{split}
\|v_h-&L_{\Omega\cup\{a_1,\dots,a_N\}}v_h\|_{L^\infty}\\
&\leq C
\left(h^{2/3}\right)^{\frac{2\alpha}{4\alpha+3}}
\left(h^{-1/3}\right)^{\frac{3}{4\alpha+3}}\\
&\leq C h^{\frac{4\alpha-3}{3(4\alpha+3)}}\,.
\end{split}
\]
\end{proof}

\begin{lemma}
\label{lem:vhbound}
  Assume that $v_h\in W^{2,2}(\Omega)$  with 
$v_h = 0$ on $\partial \Omega$ satisfying \eqref{eq:27}. Then for $i=1,\dots,N$,
\[
|v_h(a_i)|\leq C\,.
\]

\end{lemma}

\begin{proof}
  For $\rho>0$, let $\Omega_\rho:=\{x\in \Omega:\dist(x,\partial\Omega)>\rho\}$. We estimate
\[
  \begin{split}
|v_h(a_i)|&\leq \left|L_{\partial\Omega_\rho}v_h(a_i)\right|+\left|L_{\partial\Omega_\rho}v_h(a_i)-v_h(a_i)\right|\\
&\leq \sup_{x\in\partial\Omega_\rho} |v_h|+\sup_{x\in \Omega_\rho}\left|L_{\partial\Omega_\rho}v_h(x)-v_h(x)\right|\,.
\end{split}
\]
Using again the fact that $\|v_h\|_{C^{0,\alpha}(\Omega)}\leq C(\alpha) h^{-1/3}$ for every $\alpha\in (0,1)$ by the continuous embedding $W^{2,2}(\Omega)\hookrightarrow C^{0,\alpha}(\Omega)$, we obtain
\[ 
\sup_{x\in\partial\Omega_\rho} |v_h|\leq C \rho^\alpha h^{-1/3}\,.
\]
By Lemma \ref{lemma: Alexandrov}, 
\[
\sup_{x\in \Omega_\rho}\left|L_{\partial\Omega_\rho}v_h(x)-v_h(x)\right|\leq C \mu_{v_h}(\Omega_\rho)^{1/2}\,.
\]
Let $\phi_\rho\in C^\infty_c(\Omega)$ such that $\phi_\rho=1$ on $\Omega_\rho$ and such that $\|\phi_\rho\|_{W^{2,2}_0(\Omega)}\leq C\rho^{-3/2}$ (see the  proof of Lemma \ref{lem:c0alphaestim} for the construction). We have that
\[
  \begin{split}
\mu_{v_h}(\Omega_\rho)&\leq\int\phi_\rho\d \mu_{v_h}\\
&\leq 
\int\phi_\rho\d\overline\mu+\|\mu_{v_h}-\overline\mu\|_{W^{-2,2}(\Omega)}\|\phi_\rho\|_{W^{2,2}_0(\Omega)}\\
&\leq \overline\mu(\Omega)+ Ch^{2/3}\rho^{-3/2}
\end{split}
\]
and inserting this last estimates in the previous ones, we obtain
\[
|v_h(a_i)|\leq C(\overline\mu,\Omega)+C(\alpha)\rho^{\alpha}h^{-1/3}+Ch^{2/3}\rho^{-3/2}\,.
\]
Choosing $\alpha=1-\e$ for some small $\e>0$, and $\rho=h^{1/3+\e}$, we obtain the claim of the present lemma.  
\end{proof}

Using the  lemmata above, we arrive at the suboptimal global estimate comparing low-energy competitors and the solution to \eqref{eq: MAD}:
\begin{proposition}
\label{prop:suboptimal}
  Let $\Omega \subseteq \R^2$  and $\overline \mu=\sum_{i=1}^N \sigma_i \delta_{a_i}$ be as in the statement of Theorem \ref{thm:main}, $v_h$ a sequence satisfying the upper bound \eqref{eq:27}, $\alpha\in (\frac34,1)$, 
and $h_0,C,\gamma=\gamma(\alpha)$ chosen as in  Lemma \ref{lem:lifting_comparison}. Then 
\[
\|v_h-v_0\|_{L^\infty}\leq C h^{\gamma}\quad\text{ for } h<h_0\,.
\]

\end{proposition}

\begin{proof}
Let $r= \frac14 \min_{i,j}\min(|a_i-a_j|,\dist(a_i,\partial \Omega)) >0$. Find for each $i=1,\ldots,N$ two test functions $\phi_i,\Phi_i\in  C_c^\infty(\Omega)$ with $\phi_i(a_j) = \Phi_i(a_j) = \delta_{ij}$ and $0\leq \phi_i \leq \mathds{1}_{B(a_i,r)} \leq \Phi_i$. Then denoting  by $\langle\cdot,\cdot\rangle$ the dual pairing  $W^{-2,2}\times W^{2,2}_0\to\R$,  $|\langle \mu_{v_h}-\overline\mu,\phi_i\rangle|+|\langle \mu_{v_h}-\overline\mu,\Phi_i\rangle|\lesssim h^{2/3}$ by the assumed estimates on $I_h(v_h)$, and hence
  \begin{equation}\label{eq:12}
    \begin{split}
 \sigma_i-Ch^{2/3}&\leq \langle \mu_{v_h},\phi_i \rangle\\
& \leq    \mu_{v_h}(B(a_i,r))\\
& \leq \langle \mu_{v_h},\Phi_i \rangle \\
&\leq \sigma_i+Ch^{2/3}\,.
\end{split}
\end{equation}
Let us write
\[
w_h:=L_{\partial \Omega \cup \{a_1,\ldots,a_N\}}v_h\,.
\]
For every $i\in\{1,\dots,N\}$, $\partial^-w_h(a_i)$ is convex. Furthermore $|v_h(a_i)|\leq C$ by Lemma \ref{lem:vhbound} and hence 
  \begin{equation}\label{eq:25}
    \diam \partial^-w_h(a_i)\leq C
  \end{equation}
by the definition of $w_h$. 
Combining \eqref{eq:25} with  the convexity of $\partial^-w_h(a_i)$, we get that there exists $\e_0$ such that 
  \begin{equation}\label{eq:26}
    \L^2\left(\{p\in \R^2:\dist(p,\partial(\partial^-w_h(a_i))<\e\})\right) \leq C\e
  \end{equation}
for every $\e<\e_0$.

Now let  $p\in \partial^-w_h(a_i)$ with $\dist(p,\partial(\partial^-w_h(a_i))) \geq \frac{\|v_h -  w_h\|_{L^\infty}}{r}$. Then for any $x\in \overline \Omega$ with $|x-a_i|> r$ we may choose $p'\in \partial B\left(p,\frac{\|v_h-w_h\|_{L^\infty}}{r}\right)$ such that $(p-p')\cdot (x-a_i)=-|p-p'||x-a_i|$. Thus 
\[
    \begin{split}
p\cdot x - v_h(x)  &\leq p\cdot x - w_h(x) + \|v_h -  w_h\|_{L^\infty}\\
&=p'\cdot x-w_h(x)+(p-p')\cdot x+\|v_h -  w_h\|_{L^\infty}\\
&\leq p'\cdot a_i-w_h(a_i)+(p-p')\cdot x+\|v_h -  w_h\|_{L^\infty}\\
&\leq p\cdot a_i-w_h(a_i)-|p-p'||x-a_i|+\|v_h -  w_h\|_{L^\infty}\\
& < p\cdot a_i - v_h(a_i),
\end{split}
\]
which implies that 
\[
p\not\in \partial^-v_h(x)\,, 
\]
and hence 
  \begin{equation*}
\left\{p\in \partial^-w_h(a_i):\dist(p,\partial(\partial^-w_h(a_i)))>\frac{\|v_h -  w_h\|_{L^\infty}}{r}\right\}\subseteq \partial^- v_h(B(a_i,r))\,. 
\end{equation*}
From this inclusion  we now obtain using
\eqref{eq:12}, \eqref{eq:26} and Lemma \ref{lem:lifting_comparison}, 
\[
  \begin{split}
\mu_{w_h}(\{a_i\})&=\L^2(\partial^-w_h(a_i))\\
&\leq \L^2(\partial^- v_h(B(a_i,r))+C\|v_h-w_h\|_{L^\infty}\\
&\leq \sigma_i+C h^{2/3}\left(\log \frac{1}{h}\right)^{1/6}+C h^{\gamma}\,.
\end{split}
\]

Next we prove
  \begin{equation}\label{eq:8}
    \partial^-v_h(B(a_i,r))\subseteq \left\{p:\dist(p,\partial^-w_h(a_i))<\frac{4\|v_h-w_h\|_{L^\infty}}{r}\right\}\,.
  \end{equation}
To show \eqref{eq:8}, observe that $p\in \partial^-w_h(a_i)$ if and only if there exists an affine function $A$ with $\nabla A=p$, $w_h\geq A$ on $B(a_i,r)$ and  $w_h=A$ on at least one point in $B(a_i,r)$.
For any $x_0\in B(a,r_i)$ and $p\in \partial^-v_h(x_0)$, we have 
\[
v_h(x_0)+p\cdot (x-x_0)\leq v_h(x)\leq w_h(x) \quad\text{ for } x\in  B(a_i,2r)\,. 
\]
Assume $v_h(x_0)+p\cdot (a_i-x_0)<w_h(a_i)-4\|v_h-w_h\|_{L^\infty}$. Then with $x_1=2x_0-a_i$ and the fact that $w_h$ is affine on the line segment $[a_1x_1]$, we get $v_h(x_1)\geq v_h(x_0)+p\cdot(x_1-x_0)\geq  w_h(x_1)+2\|v_h-w_h\|_{L^\infty}$, which is a contradiction. Hence
\[
   w_h(a_i)+p\cdot (x-a_i)\leq w_h(x)+4\|v_h-w_h\|_{L^\infty}\text{ for } x\in B(a_i,r)\,.
\]
Thus one may find $p'\in  B\left(p,\frac{4\|v_h-w_h\|_{L^\infty}}{r}\right)$ that 
satisfies 
\[
w_h(a_i)+p'\cdot (x-a_i)\leq w_h(x)\quad\text{ for } x\in  B(a_i,r)\,,
\]
which implies $p'\in \partial^-w_h(a_i)$ and hence \eqref{eq:8}.
From \eqref{eq:8} and  \eqref{eq:12}, we obtain
\[
  \begin{split}
\mu_{w_h}(\{a_i\})&=\L^2(\partial^-w_h(a_i))\\
&\geq \L^2(\partial^- v_h(B(a_i,r))-C\|v_h-w_h\|_{L^\infty}\\
&\geq \sigma_i-C h^{2/3}-C h^{\gamma}\,.
\end{split}
\]
Hence we have proved
\[
|\mu_{w_h}(\{a_i\})-\sigma_i|\leq C h^\gamma \quad \text{ for }i=1,\dots,N\,.
\]
By Lemma \ref{lem:stability}, we obtain $\|w_h-v_0\|_{L^\infty}\leq C h^\gamma$. Thus
\[
\|v_h-v_0\|_{L^\infty}\leq \|w_h-v_0\|_{L^\infty}+\|w_h-v_h\|_{L^\infty}\leq C h^\gamma\,.
\]
\end{proof}

\subsection{Estimates for nearly conical configurations}
\label{sec:estim-nearly-conic}
In this subsection we consider a closed convex set $E$ with $0\in E$, and $|x|\simeq 1$ for $x\in \partial E$, and 
a convex function $v\in C^0(\overline E)$ with $v=0$ on $\partial E$.
We are going to define a weighted integral of the Hessian determinant that can serve as a replacement for the energy $\|\mu_v-\delta_0\|_{W^{-2,2}}$,  
  \begin{equation}\label{eq:45}
    F(v)= \int_E|x|\d \mu_v\,.
  \end{equation}

\begin{proposition}
  \label{prop:conical_estimate}
  Let $E\subseteq \R^n$ be a convex open set with $0\in E$, with $1/R\leq |x|\leq R$ for $x\in \partial E$, and $U\subseteq\R^n$ open with $E\cc U$. Let $v\in C^0(U)$ be convex with $v=0$ on $\partial E$ and $\mu_v(\partial E)=0$, and $\hat v:E\to \R$ defined by 
\[{\hat v} = L_{\partial E \cup \{0\}}v\,.
\]
Then
  \begin{equation}\label{eq:24}
{\hat v}(x)- v(x)\leq  C(R,n) \min(|x|^{(n-1)/n} F( v)^{1/n},|v(0)||x|)\quad\text{ for all }x\in E\,.
\end{equation}
\end{proposition}

\begin{proof}
First we assume that the restriction of $v$ to $E$ is in $W^{2,\infty}$, $v\in W^{2,\infty}(E)$. We may write by a change of variables
\begin{equation}
  \begin{split}
  F(v) &=\int_E  |x| \det \nabla^2 v\d x \\
&= \int_{\partial^- v(E)}|(\nabla v)^{-1}(p)|\,\d p \\
&\geq \int_{\partial^-{\hat v}(0)}|(\nabla v)^{-1}(p)|\,\d p.
\end{split}
\end{equation}

In the last step we have used the monotonicity of the subdifferential (see \cite[Lemma 2.7]{figalli2017monge} or Proposition \ref{prop: comparison} with $\phi=1$). The point $(\nabla v)^{-1}(p)\in E$ can be written for almost every $p\in \partial^-v(E)$ as
\begin{equation}
  (\nabla v)^{-1}(p) = \argmax_{x\in E} p\cdot x - v(x).
\end{equation}

Let
\begin{equation}x_0 \in \argmax_{x\in E\setminus\{0\}}\frac{{\hat v}(x)-v(x)}{|x|^{(n-1)/n}}
\end{equation}
and $y_0 := {\hat v}(x_0)-v(x_0)$. (A maximizer exists since $v,{\hat v}$ are Lipschitz at $0$ and coincide on $\partial E$.) Let $p_0 \in \partial^-{\hat v}(x_0)$. We claim that
\begin{equation}\label{eq: inverse gradient norm}
|(\nabla v)^{-1}(p)| \geq \frac1{2}|x_0|\quad\text{ for all }p\in \partial^-{\hat v}(0)\cap B\left(p_0, \frac{y_0}{2|x_0|}\right).
\end{equation}
Since $\L^n(\partial^-{\hat v}(0)\cap B(p_0, r)) \geq cr^n$ for all $r\leq |v(0)|$, and $r = \frac{y_0}{2|x_0|}\leq C|v(0)|$ by the Lipschitz continuity of $v,{\hat v}$, we obtain
\begin{equation*}
  \begin{split}
F(v) &\geq \int_{\partial^-{\hat v}(0)\cap B(p_0, \frac{y_0}{2|x_0|})} |(\nabla v)^{-1}(p)|\,\d p \\
&\geq \frac{|x_0|}{2} c\left(\frac{y_0}{|x_0|}\right)^n \\
&= c \frac{y_0^n}{|x_0|^{n-1}} \\
&= c\max_{x\in E} \frac{({\hat v}(x)-v(x))^n}{|x|^{n-1}}.
\end{split}
\end{equation*}

Reordering yields
\begin{equation*}
{\hat v}(x) - v(x) \leq C|x|^{(n-1)/n}F(v)^{1/n}\quad \text{ for all }x\in E.
\end{equation*}
The inequality ${\hat v}(x) - v(x) \leq C|v(0)||x|$ follows from the uniform Lipschitz bounds near $0$.

Now to show \eqref{eq: inverse gradient norm}:

Let $p\in \partial^- {\hat v}(0)$. Then ${\hat v}(x) \geq p\cdot x$ for all $x\in E$ and for $|x| < \frac{|x_0|}{4}$ we have
\begin{equation}
p\cdot x - v(x) \leq {\hat v}(x) - v(x) \leq \frac{|x|^{1/2}}{|x_0|^{1/2}} y_0 < \frac{y_0}{2}.
\end{equation}

On the other hand for $|p-p_0|\leq \frac{y_0}{2|x_0|}$ we have since $p_0\cdot x_0 = {\hat v}(x_0)$
\begin{equation}
p\cdot x_0 - v(x_0) = p_0 \cdot x_0 - v(x_0) + (p-p_0)\cdot x_0 \geq y_0 - |p-p_0||x_0| \geq \frac{y_0}{2}.
\end{equation}
In particular $|(\nabla v)^{-1}(p)|\geq \frac{|x_0|}{4}$. This completes the proof for $v\in W^{2,\infty}(E)$. 

\medskip

In the general case,  consider the approximations $v_\e:=P_\e v|_U\in W^{2,\infty}(U)$ of $ v$  given obtained by applying the smoothing operation $P_\e$ from Lemma \ref{lem:convex_smoothing}. Let $E_\e:=\{x\in U:  v_\e<0\}$. Then $\mu_{ v_\e}\wsto \mu_{ v}$ weakly-* as measures on $\overline E$ and $\mathds{1}_{E_\e}\to \mathds{1}_{E}$ in $L^1(U)$. Using the fact $\mu_v(\partial E)=0$, 
we obtain 
\[
  \begin{split}
\int_{E_\e}|x|\d\mu_{v_\e}&\to  \int_E|x|\d\mu_{v}\\
\hat v_\e(x)-v_\e(x)&\to \hat v(x)-v(x) \quad \text{ for every }x\in E\,.
\end{split}
\]
This finishes the proof.
\end{proof}

\begin{remark}
\label{rem:show_optimal}
The scaling exponent $\frac{1}{n}$ with which the energy enters the right hand side in \eqref{eq:24} is optimal in the sense that one may construct a family of  examples that satisfies this scaling. To alleviate the notation, we assume  $n=2$, but  the general case can be treated in analogous fashion.

\medskip

Consider $E=[-1,1]^2$, and for $0<R<H<1$, define  the piecewise affine function
\[
  \begin{split}
v_{R,H}(x)
=\{\sup w(x): w\text{ convex },w|_{\partial E}=0, w(0)=-1, w(Re_1) = R-H\}\,.
\end{split}
\]
We observe that with $\hat v_{R,H}=L_{\partial E\cup\{0\}}v_{R,H}$ defined as in the proposition, we have that
 $\hat v_{R,H} = v_{0,0} = \max(|x_1|,|x_2|)$. We have that
\[
F(v_{R,H}) =\int |x|\d\mu_{v_{R,H}}(x)= R\L^2(\partial^-v_{R,H}(Re_1))\,,  
\]
and calculate
\[
\partial^- v_{R,H}(Re_1) = \conv\left(\left(1-\frac{H}{R},\frac{H}{R}\right),\left(1-\frac{H}{R},-\frac{H}{R}\right),\left(\frac{1+H}{1-R},0\right)\right)\,,
\]
so that
\[
F(v_{R,H}) \geq R \left(\frac{H}{R}\right)^2 = \frac{H^2}{R},  
\]
and of course
\[
\hat v_{R,H}(Re_1) - v_{R,H}(Re_1) = H \geq \sqrt{F(v_{R,H})R}.
\]
This proves  our claim.   
\end{remark}

\subsection{Quantitative behavior near ridges}
\label{sec:ridges}

Before demonstrating how to ``transfer the optimal estimates from the boundary to the bulk''  in Proposition \ref{prop: ridge} below,  which is the pivotal step in obtaining control over $v_h$ in the vicinity of the ridges, we establish some helpful notation in the 
forthcoming lemma. 
This lemma introduces  affine maps $\ell^k$ ($\ell_{ij}^k$ in the statement of the lemma). The maps $\tilde \ell^k$  mentioned  at the outset of the present  section are perturbations of the maps $\ell^k$, see \eqref{eq:34} below. The maps $\tilde\ell^k$ will act as auxiliary tools facilitating the utilization of Proposition \ref{prop:conical_estimate} in proximity to the vertices.

\begin{lemma}
\label{lem:ridgeaux}
Let $v_0$ be the solution of \eqref{eq: MAD}. Then the choice of points  
$b_{ij}^+,b_{ij}^-\in\Omega$ for $(i,j)\in\rid$ from Lemma \ref{lemma: 2d ridge} may be modified in  a way such that, additionally to the properties claimed in that lemma,  there exist  
affine functions $\ell_{ij}^i,\ell_{ij}^j$ for  $(i,j)\in \rid$  with the following properties (see Figure \ref{fig:singlevertexcut} for an illustration):
\begin{itemize}
\item[(i)] For every $(i,j)\in \rid$,
\begin{alignat*}{3}
v_0(a_k)&< \ell_{ij}^k(a_k) \quad&&\text {for }k\in\{i,j\}\\
v_0(a_l)&>\ell_{ij}^k(a_l) \quad&&\text{for } k\in\{i,j\},l\in \{1,\dots,N\}\setminus\{k\}\\
\ell_{ij}^k(b_{ij}^l)&=v_0(x) \quad &&\text{for } k\in \{i,j\}, l\in \{\pm\}\,.
\end{alignat*}

\item[(ii)] For every $(i,j)\in \rid$, with  
  \begin{equation}\label{eq:3}
K_{ij}:=\conv\left(R_{ij}\cup \{x:v_0(x)\leq\ell_{ij}^i\}\cup \{x:v_0(x)\leq\ell_{ij}^j\}\right)
\end{equation}
we have that  
\[
a_l\not\in K_{ij}\text{  for }l\not\in\{i,j\}\,.
\]

\end{itemize}
\end{lemma}

\begin{proof}
Fix $(i,j)\in\rid$.
We initially assume that $b_{ij}^\pm$ are as in the statement of Lemma \ref{lemma: 2d ridge}. By an affine change of variables, we may assume $[a_ia_j]=[-L,L]\times\{0\}$, $b_{ij}^\pm=(0,\pm\beta)$ for some $L,\beta>0$ and 
\[
  \begin{split}
\nabla v_0&=p^+ \text{ on } [a_ib_{ij}^+a_j] \\
\nabla v_0&=p^- \text{ on } [a_ib_{ij}^-a_j] 
\end{split}
\]
with $p^+\neq p^-$. Clearly $\overline p:=\frac12(p^++p^-)\in \partial^-v_0(a_i)\cap \partial^-v_0(a_j)$. Furthermore, the affine function 
\[
\ell_{\overline p, d}(x)=v(a_i)+d+(x-a_i)\cdot\overline p=v(a_j)+(x-a_j)\cdot \overline p+d
\]
satisfies, for $d>0$ small enough, 
  \begin{alignat*}{3}
\ell_{\overline p, d}(a_l)&=v(a_l)\quad &&\text{ for } l\in \{i,j\}\,,\\
\ell_{\overline p, d}(a_l)&<v(a_l)\quad &&\text{ for } l\in\{1,\dots,N\}\setminus \{i,j\}\,.
\end{alignat*}
Without loss of generality, for $\alpha>0$ small enough, we may assume that
\[
  \begin{split}
p_i(\alpha)&:=\overline p+\alpha (p^+-p^-)^\bot\in \left(\partial^-v_0(a_i)\right)^\circ\,,\\
\text{ and }\quad p_j(\alpha)&:=\overline p-\alpha (p^+-p^-)^\bot\in \left(\partial^-v_0(a_j)\right)^\circ\,.
\end{split}
\]
Now define $d(\alpha)>0$ by requiring
that the affine functions 
\[
  \begin{split}
\ell_{p_i(\alpha),d(\alpha)}&:=v_0(a_i)+d(\alpha)+p_i(\alpha)\cdot (x-a_i)\\
\ell_{p_j(\alpha),d(\alpha)}&:=v_0(a_j)+d(\alpha)+p_j(\alpha)\cdot (x-a_j)
\end{split}
\]
satisfy 
\[
  \begin{split}
\ell_{p_i(\alpha), d(\alpha)}\left((+ L/2,0)\right)&=v_0\left((+ L/2,0)\right)\\
\ell_{p_i(\alpha), d(\alpha)}\left((- L/2,0)\right)&=v_0\left((- L/2,0)\right)\,.
\end{split}
\]
(Note that $(L/2,0)=\frac34 a_j+\frac14 a_i$, $(-L/2,0)=\frac34 a_j+\frac14 a_i$.)
For $\alpha$ small enough and $k\in \{i,j\}$, this implies  
    \begin{alignat*}{3}
\ell_{p_k(\alpha), d(\alpha)}(a_k)&>v_0(a_k)\,, &&\\
\ell_{p_k(\alpha), d(\alpha)}(a_l)&<v_0(a_l)\quad &&\text{ for } l\in \{1,\dots,N\}\setminus\{k\}\,.
\end{alignat*}
Moreover, it implies that there exists $b(\alpha)>0$, $b$ continuous with $\lim_{\alpha\to 0}b(\alpha)=0$, such that 
\[
\ell_{p_i(\alpha), d(\alpha)}\left((0,\pm b(\alpha))\right)=v_0\left((0,\pm b(\alpha))\right)=\ell_{p_j(\alpha), d(\alpha)}\left((0,\pm b(\alpha))\right)\,.
\]
Choosing some small positive $\alpha$ and replacing $b_{ij}^\pm$ by $(0,\pm b(\alpha))$ completes the construction, setting $\ell_{ij}^k=\ell_{p_k(\alpha), d(\alpha)}$ for $k\in\{i,j\}$.
\end{proof}

In the statement of the following proposition, we assume that $\Omega\subseteq\R^2$ and $\overline\mu=\sum_{i=1}^N\sigma_i\delta_{a_i}$ are as in Theorem \ref{thm:main}, $v_0$ is the solution of \eqref{eq: MAD}, and $(i,j)\in\rid$. Define $U_{ij}\subseteq \Omega$ to be an open convex neighborhood of $K_{ij}$ defined in \eqref{eq:3} such that 
\[
a_l\not\in \overline{U_{ij}} \text{ for }l\in \{1,\dots,N\}\setminus\{i,j\}\,.
\]

\begin{proposition}\label{prop: ridge}
 There exist  constants $\eps_0, C>0$  such that for  every convex $v\in C^0(U_{ij})$ with $\|v-v_0\|_{L^\infty(U_{ij})}\leq \eps_0$, 
  \begin{equation}
    |v(x) - L_{\{a_j, b_{ij}^+,a_i, b_{ij}^-\}}v(x)| \leq C\sqrt{F(v)}\sqrt{F(v) + \dist(x,[a_ia_j])} \quad\text{ for } x\in R_{ij}\,,
  \end{equation}
   where $F(v) = \int_{U_{ij}}\dist(\cdot,\{a_i,a_j\})\d \mu_v$.
\end{proposition}

\begin{proof}
Without loss of generality we assume that $a_j = -a_i  = L e_1$, $L > 0$. Let $\ell_{ij}^i,\ell_{ij}^j$ be  the affine functions  introduced in Lemma \ref{lem:ridgeaux}. We will 
write $b^\pm\equiv b_{ij}^\pm$, $U\equiv U_{ij}$, $R\equiv R_{ij}$, $\ell^k\equiv \ell_{ij}^k$.

\emph{Step 1: Conical structure of $v_0$.}

For $k\in\{i,j\}$, let 
\[
E_0^k := \{x:v_0(x)\leq \ell^k(x)\}\,.
\]
By the properties of $\ell^i,\ell^j$ stated in Lemma \ref{lem:ridgeaux} we have that 
\begin{equation}\label{eq: E^0}
  \begin{aligned}
  a_k\in E_0^k, \quad  \dist(a_k,\partial E_0^k) \gtrsim 1,\quad \dist(a_k, E_0^l) \gtrsim 1\\
    \dist(E_0^k, \partial U) \gtrsim 1, \quad\dist (\partial E_0^i\cap\partial E_0^j,[a_ia_j])\gtrsim 1,\quad\L^2(E_0^i,E_0^j) \gtrsim 1.
  \end{aligned}
\end{equation}  The implicit constants in the above  inequalities depend only on $\Omega,U$ and $\overline\mu$.



\emph{Step 2: Near-conical structure of $v$.}

For $k\in\{i,j\}$, define the affine functions $\tilde \ell^i,\tilde \ell^j$ by requiring
  \begin{equation}\label{eq:34}
  \begin{split}
    \tilde \ell^k(b^\pm)&=v(b^\pm)\\
    \tilde \ell^k(a_k)&=\ell^k(a_k)
  \end{split}
\end{equation}
for $k\in \{i,j\}$. Note that this choice implies $\|\tilde \ell^k-\ell^k\|_{L^\infty}\leq C\e_0$. Hence, defining 
convex sets $E^k \subseteq \Omega$ through 
\[
E^k:= \{x\in \Omega\,:\,v(x)\leq \tilde \ell^k(x)\}\,,
\] 
we obtain that 
 \eqref{eq: E^0} also holds with $E_0^k$ replaced by $E^k$ provided that $\e_0$ is small enough, the choice depending only on $v_0$, and hence on $\Omega,\overline\mu$.
We now apply Proposition \ref{prop:conical_estimate} to $v-\tilde \ell^k$ in the two sets $E^k$, $k\in\{i,j\}$, where for the application of that proposition we translate $a_k$ into the origin. This yields two convex functions  $\hat g_k:E^k\to\R$ with $\hat g_k=0$ on $\partial E^k$ and $\hat g_k(a_k)=v(a_k)-\tilde \ell^k(a_k)$, $k\in \{i,j\}$. We set 
\[
\hat v_{k} := \hat g_{k}+\tilde \ell^{k}\,.
\]
By Proposition \ref{prop:conical_estimate} we have 
\begin{equation*}
  \begin{split}
|v(x)-\hat v_{k}(x)|&=|(v-\tilde\ell^k)-\hat g_k|\\
& \leq C \sqrt{ \int_{E^k} |x'-a_{k}|\d\mu_{v}(x')}  \sqrt{|x-a_{k}|}\quad\text{ for }x\in E^{k}\,.
\end{split}
\end{equation*}
Using $|\cdot-a_{k}|\leq C\dist(\cdot,\{a_{i},a_j\})$ on $E^{k}$, we deduce 
\begin{equation}\label{eq: conical}
|v(x)-\hat v_{k}(x)| \leq C \sqrt{F(v)}  \sqrt{|x-a_{k}|}\quad\text{ for }x\in E^{k}\,.
\end{equation}
We define  the  function 
\[
w:=L_{\{a_j,b_+,a_i,b_-\}}v
\]
which is affine on $[a_j b_+a_i]$ and on $[a_j b_-a_i]$. For this choice of $w$, \eqref{eq: conical} implies 
\begin{equation}\label{eq: boundary conical}
|v(x)-w(x)|\leq C\sqrt{F(v)}\sqrt{|x_2|} \quad\text{ for all }x\in\partial R\,,
\end{equation}
since $w=\hat v_i$ and $|x-a_i|\lesssim |x_2|$ on $[a_ib_{ij}^+]\cup[a_ib_{ij}^-]$ with analogous estimates with the roles of $i$ and $j$ reversed. 

\emph{Step 3: Boundary comparison function.}
We will extend a weaker version of \eqref{eq: boundary conical} to the interior of $R$, losing some control near the ridge $[a_ia_j] \subseteq \{x_2=0\}$.
We define the comparison function
\begin{equation}\label{eq: w0}
w_0(x):= w(x) - C\sqrt{F(v)}\sqrt{\alpha F(v) + |x_2|},
\end{equation}
where $\alpha>0$ is chosen large enough so that $w_0:\R^2\to \R$ is convex with 
\[
|[\partial_2 w_0]|\geq |[\partial_2 w]|/2\quad\text{ on }\{x_2=0\}\,,
\]
where $[\partial_2 w_0]$ denotes the jump of  $\partial_2 w_0$ at $\{x:x_2=0\}$.
We note that such a choice of  $\alpha$ is possible and depends only  on $v_0$. We also note that by \eqref{eq: boundary conical}, 
\[w_0\leq v\leq w\text{  on }\partial R\,.
\]
Of course, $w_0(x_1,x_2)$ only depends on $x_2$. For later reference, we note the following inequality for arguments $x_2,y_2$ with opposite signs, that follows from the  convexity of $w_0$, 
  \begin{equation}\label{eq:13}
    \begin{split}
      w_0(t,y_2)-\lim_{x_2\downarrow 0} \left(w_0(t,x_2)+\partial_2 w_0(t,x_2)(y_2-x_2)\right)&=w_0(t,y_2)-w_0(t,0)-y_2\partial_2w_0(t,0+)\\
&\geq |[\partial_2 w_0]| |y_2|\quad\text{ for }y_2<0\,,
    \end{split}
  \end{equation}
with an obvious analogue for $y_2>0$ and the limit $x_2\uparrow 0$.

\emph{Step 4: Upper bound for $w_0 - v$.}
We now show that for all $x=(x_1,x_2)\in R$, we have
\begin{equation}\label{eq: w0 - v}
  v(x) \geq w_0(x) - C\sqrt{F(v)}\sqrt{F(v)+|x_2|}.
\end{equation}

To do so, consider the non-negative concave quadratic test function
\[
\phi(x) := ((x-a_i)\cdot e_1)((a_j-x)\cdot e_1)\geq 0\quad \text{ for } x\in R\,,
\]
which is constant on all vertical lines and satisfies 
  \begin{equation}\label{eq:6}
    \frac{1}{C}\dist(x,\{a_i,a_j\})\leq \phi(x) \leq  C \dist(x,\{a_i,a_j\})\text{ on }R\,.
  \end{equation}
  
Next we will apply the monotonicity formula in Proposition \ref{prop: comparison} to the functions $v\leq L_{\partial R \cup\{x\}}v$ which coincide on $\partial R$.  The application of the proposition yields
\begin{equation}\label{eq: comparison applied}
   \phi(x)\L^2(\partial^- L_{\partial R \cup\{x\}}v(x))=\int\phi\d\mu_{L_{\partial R \cup\{x\}}v} \leq \int\phi\d \mu_v \lesssim F(v).
\end{equation}

If $v(x)\geq w_0(x)$, \eqref{eq: w0 - v} holds. Otherwise, we claim that $\partial^- L_{\partial R \cup\{x\}}v(x)$ has area bounded from below by
\begin{equation}\label{eq: central subgradient}
  \L^2(\partial^- L_{\partial R \cup\{x\}}v(x)) \gtrsim\min\left(\frac{(w_0(x)-v(x))^2}{\dist(x,\{a_i,a_j\})|x_2|}, \frac{w_0(x) - v(x)}{\dist(x,\{a_i,a_j\})}\right).
\end{equation}

To prove \eqref{eq: central subgradient}, take $p_0\in \partial^- w_0(x)$. 
Write 
\[
(-L,0)=a_i=a_-\,,\quad(+L,0)= a_j=a_+
\]
and set 
\[
p_1 = p_0 - \sgn(x_1)\frac{w_0(x)-v(x)}{|a_{\sgn(x_1)}-x|}e_1\,.
\]
We note 
\[
- \sgn(x_1)\frac{e_1\cdot (y-x)}{|a_{\sgn(x_1)}-x|}\leq 1 \quad\text{ for all }y\in \partial R\,.
\]
Thanks to this observation we obtain for every $y\in \partial R$ the following chain of estimates:
  \begin{equation}\label{eq:2}
    \begin{split}
      p_1\cdot y-v(y)&=p_0\cdot y-w_0(y)+(p_1-p_0)\cdot y+w_0(y)-v(y)\\
                     &\leq p_0\cdot x-w_0(x)+(p_1-p_0)\cdot y\\
                     &\leq p_1\cdot x-w_0(x)+(p_1-p_0)\cdot (y-x)\\
                     &=p_1\cdot x- w_0(x)-(w_0(x)-v(x))\frac{(y-x)\cdot e_1 \sgn(x_1)}{|a_{\sgn(x_1)}-x|}\\
                     &\leq p_1\cdot x-v(x)\,,
    \end{split}
  \end{equation}
where we have used $p_0\in \partial^-w_0(x)$ and $w_0(y)\leq v(y)$ for all $y\in\partial R$ to obtain the first inequality. The estimate \eqref{eq:2} implies $p_1\in \partial^- L_{\partial R \cup\{x\}}v(x)$.

In a similar way, we set  
\[
p_2 = p_0 - \sgn(x_2)\min\left(\frac{w_0(x)-v(x)}{|x_2|}, \frac12|[\partial_2 w_0]|\right)e_2\,.
\]
We estimate assuming $y\in \partial R$, this time keeping the non-positive term $r(x_2,y_2):=-w_0(y)+w_0(x)+(y-x)\cdot p_2$, whose analogue we discarded in the previous calculation: 
  \begin{equation}\label{eq:17}
    \begin{split}
      p_2\cdot y-v(y)&\leq p_2\cdot x-w_0(x)+(p_2-p_0)\cdot (y-x)+r(x_2,y_2)\\
                     &\leq p_2\cdot x-w_0(x)-\sgn(x_2)\min\left(\frac{w_0(x)-v(x)}{|x_2|}, \frac12|[\partial_2 w_0]|\right)e_2\cdot (y-x)\\
&\quad +r(x_2,y_2)\,.
    \end{split}
  \end{equation}
In order to derive from \eqref{eq:17}  the wished for estimate 
\begin{equation}\label{eq:16}
  p_2\cdot y-v(y)\leq p_2\cdot x-v(x)\,,
\end{equation}
 we will assume  that $x_2\geq 0$, the proof for $x_2\leq 0$ working out in an analogous fashion. We  make a case distinction: If $y_2\geq x_2$, then of course
\[-\min\left(\frac{w_0(x)-v(x)}{|x_2|}, \frac12|[\partial_2 w_0]|\right)e_2\cdot (y-x)\leq 0\,,\]
and \eqref{eq:16} follows from \eqref{eq:17} and $r(x_2,y_2)\leq 0$. If $0\leq y_2\leq x_2$ then 
\[
- \min\left(\frac{w_0(x)-v(x)}{|x_2|}, \frac12|[\partial_2 w_0]|\right)e_2\cdot (y-x)\leq w_0(x)-v(x)\,,
\]
and again \eqref{eq:16} follows from \eqref{eq:17} and $r(x_2,y_2)\leq 0$. Finally, for the case $y_2<0$ we note that by the convexity of $w_0$ and \eqref{eq:13},  
\[
r(x_2,y_2)\leq \lim_{s\downarrow 0} r(s,y_2)\equiv r(0+,y_2)\leq -|[\partial_2 w_0]||y_2|\,.
\]
Hence 
\[
  \begin{split}
    -\min\left(\frac{w_0(x)-v(x)}{|x_2|}, \frac12|[\partial_2 w_0]|\right)e_2\cdot (y-x)&\leq 
     w_0(x)-v(x)+\frac12|[\partial_2 w_0]||y_2|\\
    &\leq w_0(x)-v(x)-\frac12r(x_2,y_2)\,,
  \end{split}
\]
and again \eqref{eq:16} follows from \eqref{eq:17}. Hence we have proved
\[p_2\in \partial^-L_{\partial R \cup\{x\}}v(x)\,.\]

Summarizing, we have that 
$p_0,p_1,p_2\in \partial^-L_{\partial R \cup\{x\}}v(x)$
and hence, by the convexity of the subdifferential at a single point, also  
\[ 
\partial^-L_{\partial R \cup\{x\}}v(x) \supseteq [p_0p_1p_2]\,,
\] 
which  implies
\[
  \begin{aligned}
\L^2(  \partial^-L_{\partial R \cup\{x\}}v(x)) \geq &\frac{w_0(x)-v(x)}{|a_{\sgn(x_1)}-x|}\min\left(\frac{w_0(x)-v(x)}{|x_2|},\frac12|[\partial_2 w_0]|\right)\\
 \gtrsim& \min\left( \frac{(w_0(x) - v(x))^2}{\dist(x,\{a_i,a_j\})|x_2|}, \frac{w_0-v(x)}{\dist(x,\{a_i,a_j\})}\right),
  \end{aligned}
\]
i.e. \eqref{eq: central subgradient}. Combining \eqref{eq:6}, \eqref{eq: comparison applied} and \eqref{eq: central subgradient}  yields our claim \eqref{eq: w0 - v}.
Combining the latter  with \eqref{eq: w0} yields the statement of the proposition.
\end{proof}

\subsection{Proof of the lower bound in Theorem \ref{thm:main}}

\label{sec:lower_bound}


\begin{proof}[Proof of the lower bound in Theorem \ref{thm:main}]
By the upper bound in Theorem \ref{thm:main} and Lemma \ref{lem:W-22L2equality} we may assume that 
\[
\|\mu_{v_h}- \overline \mu\|_{W^{-2,2}(\Omega)}^2+h^2\|\nabla^2 v_h\|\leq C h^{4/3}\,,
\]
where the constant on the right hand side only depends on $\Omega,\overline\mu$. Choose $(i,j)\in\rid$, and recall the definition of the rhombus $R_{ij}$ from Lemmata \ref{lemma: 2d ridge} and \ref{lem:ridgeaux}. 
By Lemma \ref{lem:ridgeaux}, 
\[
a_l\not\in R_{ij} \quad\text{ for }l\in \{1,\dots,N\}\setminus\{i,j\}\,.
\]
We may choose two open convex sets $U^1_{ij},U^2_{ij}$ such that $R_{ij}\cc U^1_{ij}\cc U^2_{ij}$ and 
\[
a_l\not\in \overline{U_{ij}^2} \quad\text{ for }l\in \{1,\dots,N\}\setminus\{i,j\}\,.
\]
We will write $U^1\equiv U^1_{ij}$, $U^2\equiv U^2_{ij}$. 
We claim that 
\[
F(v_h):= \int_{U^1}\dist(x,\{a_1,\dots,a_N\}) \d\mu_{v_h}(x)
\]
satisfies
\begin{equation}\label{eq: claim F}
F(v_h) \leq C(\Omega,\overline \mu)h^{2/3}\left(\log\frac{1}{h}\right)^{1/2}\,.
\end{equation}
To see that \eqref{eq: claim F} holds, choose $r_i,r_j>0$ independently of $h$ such that
\[
\begin{split}
B(a_i,r_i), B(a_j,r_j)&\subseteq U^1\\
B(a_i,r_i)\cap B(a_j,r_j)&=\emptyset\,,
\end{split}
\]
for $k\in \{i,j\}$ choose $\psi_k\in C^\infty_c(B(a_k,r_k))$ such that $0\leq \psi_k\leq 1$ and $\psi_k=1$ on $B(a_k,r_k/2)$, and complete the set $\{\psi_i,\psi_j\}$ to a partition of unity $\{\psi_i,\psi_j,\varphi\}$ of $U^1$, where $\varphi\in C^\infty_c(U^2)$. Let $\eta\in C^\infty_c(\R^2)$ be a standard mollifier, and $\eta_\e(\cdot)=\frac{1}{\e^2}\eta(\cdot/\e)$ for $\e>0$. Now the $h$-dependent choice of test function we make to show \eqref{eq: claim F} is the following: We write $\e := h^{2/3}$ and set
\[
\phi_h := \left(|\cdot-a_i|*\eta_\e\right)\psi_i+\left(|\cdot-a_j|*\eta_{\e}\right)\psi_j+\varphi\,.
\]
A calculation yields
\begin{equation}
\begin{cases}
\phi_h(x) \simeq \dist(x,\{a_i,a_j\})+h^{2/3} &\text{ in }U^1\\
|\nabla \phi_h(x)|\leq C&\text{ in }U^2\\
|\nabla^2 \phi_h(x)|\leq \frac{C}{\dist(x,\{a_i,a_j\})+h^{2/3}}&\text{ in }U^2\,,
\end{cases}
\end{equation}
where all constants on the right hand side only depend on $U^1,U^2,\psi_i,\psi_j,\varphi$ (and may hence be said to depend on $\Omega,\overline\mu$).
Thus 
\[
\|\phi_h\|_{W^{2,2}_0(\Omega)}\leq C \left(\log\frac{1}{h}\right)^{1/2}
\]
and we obtain
\begin{equation}
  \begin{aligned}
F(v_h) \leq &C \int\phi_h\d \mu_{v_h}\\
& \leq \|\phi_h\|_{W^{2,2}_0(\Omega)}\|\mu_{v_h}-\overline \mu\|_{W^{-2,2}(\Omega)} +\int\phi_h\d \overline \mu \\
&\leq Ch^{2/3}\log(1/h)^{1/2}  + (\sigma_i + \sigma_j) h^{2/3}\\
& \leq Ch^{2/3}\log(1/h)^{1/2}\,,
  \end{aligned}
\end{equation}
proving \eqref{eq: claim F}.

We recall the notation $R_{ij}=[a_ib^+_{ij}a_jb^-_{ij}]$, and  assume without loss of generality that $a_i = -a_j = Le_1$. We set
\[
w_h := L_{\{a_i,a_j,b^+_{ij},b^-_{ij}\}}v_h\,.
\]
By Lemma \ref{lemma: 2d ridge}, $v_0$ is affine on $[a_ib^+_{ij}a_j]$ and on $[a_ib^-_{ij}a_j]$, with 
\[
  \begin{split}
\nabla v_0&=p^+\text{ on }[a_ib^+_{ij}a_j]\\
\nabla v_0&=p^-\text{ on }[a_ib^-_{ij}a_j]
\end{split}
\]
and $p^+\neq p^-$. 
By definition, there exist $\tilde p^+_h,\tilde p^-_h\in \R^2$ such that
\[
  \begin{split}
\nabla w_h&=\tilde p_h^+\text{ on }[a_ib^+_{ij}a_j]\\
\nabla w_h&=\tilde p_h^-\text{ on }[a_ib^-_{ij}a_j]
\end{split}
\]
and by Proposition \ref{prop:suboptimal}, we have
  \begin{equation}\label{eq:15}
|p^+-\tilde p^+_h|+|p^--\tilde p^-_h|\leq C h^\gamma
\end{equation}
for $h<h_0$. Now we may apply   
 Proposition \ref{prop: ridge} with   $U\equiv U^2$,  $\e_0\equiv C h_0^\gamma$. We note that the hypothesis $\|v-v_0\|_{L^\infty}\equiv\|v_h-v_0\|_{L^\infty}\leq C h^\gamma\leq\e_0$ is met by Proposition \ref{prop:suboptimal}, provided that $h_0$ is chosen small enough. Thus  we obtain the existence of a constant  $C_1>0$ (that only depends on  $\Omega,\overline \mu$) such that 
\begin{equation}
|v_h(x) - w_h(x)| \leq C_1\left( h^{2/3}\left(\log\frac{1}{h}\right)^{1/2} + \sqrt{h^{2/3}\left(\log\frac{1}{h}\right)^{1/2}|x_2|}\right)\quad\text{ for }x\in R_{ij}\,.
\end{equation}
Here we may  choose $C_1$ such that it additionally satisfies $C_1\geq \max(|p^+-p^-|,1)$.
In particular, setting
\[
 l(h) :=  \frac{16C_1^2h^{2/3}\left(\log\frac{1}{h}\right)^{1/2}}{|\tilde p_h^+-\tilde p_h^-|^2}\,,
\]
 we obtain,
\begin{equation}\label{eq: vh high}
  v_h(x_1,\pm l(h)) \geq w_h(x_1,\pm l(h)) -  C_1\left(1+\frac{4C_1}{|\tilde p_h^+-\tilde p_h^-|}\right)h^{2/3}\left(\log\frac{1}{h}\right)^{1/2}\,,
\end{equation}
where we have assumed that $(x_1,\pm l(h))\in R_{ij}$, which is  the case for $h$ small enough and $x_1\in [-L/2,L/2]$.
Moreover by convexity
\begin{equation}\label{eq: vh low}
  v_h(x_1,0) \leq w_h(x_1,0).
\end{equation}
Combining \eqref{eq: vh high} with \eqref{eq: vh low} and using the convexity of $v_h$, we obtain
\begin{equation}
  \begin{split}
\int_{-l(h)}^{l(h)}&\partial_2^2 v_h(x_1,x_2)\d x_2\\
 &= \partial_2 v_h(x_1,l(h)) - \partial_2 v_h(x_1,-l(h)) \\
&\geq \frac{v_h(x_1,l(h))-v_h(x_1,0)}{l(h)}+\frac{v_h(x_1,-l(h))-v_h(x_1,0)}{l(h)}\\
&\geq \frac{w_h(x_1,l(h))-w_h(x_1,0)}{l(h)}+\frac{w_h(x_1,-l(h))-w_h(x_1,0)}{l(h)}\\
&\qquad-\frac{C_1\left(1+\frac{4C_1}{|\tilde p_h^+-\tilde p_h^-|}\right)|\tilde p_h^+-\tilde p^-_h|^2}{8 C_1^2}\\
&\geq  |\tilde p_h^+-\tilde p_h^-|-\frac{|\tilde p_h^+-\tilde p_h^-|^2}{8C_1}-\frac{|\tilde p_h^+-\tilde p_h^-|}{2}\\
&\geq \frac14 |\tilde p_h^+-\tilde p_h^-|\,.
\end{split}
\end{equation}
Using  the Cauchy-Schwarz inequality, we deduce
\begin{equation}
  \begin{split}
h^2 \int_{-L/2}^{L/2} \int_{-l(h)}^{l(h)}|\partial_2^2 v_h(x_1,x_2)|^2\,\d x_2\d x_1 &\geq h^2L\frac{|\tilde p_h^+-\tilde p_h^-|^2}{8l(h)}\\
& \geq C h^{4/3} \log(1/h)^{-1/2},
\end{split}
\end{equation}
with $C>0$ depending only on  $\Omega$ and $\overline\mu$, where in the last inequality, we have used \eqref{eq:15}.
\end{proof}

\appendix

         \section{Scaling without boundary conditions: $h^2\log \frac1h$}

\label{sec:counterexample}         
         We observe that in fact the  scaling  for $M$ disclinations without any further constraints such as boundary conditions will \emph{not} be $h^{4/3}$ but $h^2\log \frac1h$, like a single disclination. 
We claim the following:
\begin{remark}
\label{rem:nobc}
Let $\Omega\subseteq\R^2$ and $\overline \mu=\sum_{i=1}^N\sigma_i\delta_{a_i}\in\M_+(\Omega)$. Then there exist  constants $C,h_0>0$ that only depend on $\Omega$ and $\overline\mu$  such that for every $h<h_0$, there exist $(\bu_h,v_h)\in W^{1,2}(\Omega;\R^2)\times W^{2,2}(\Omega)$ with 
\[
E_h(\bu_h,v_h)\leq C h^2\log \frac{1}{h}\,.
\]
\end{remark}

In the proof of the remark, we are going to use the following lemma:

\begin{lemma}
\label{lem:nobcaux}
Let  $\e\in (0,1)$ and $u_\e :\R^2\to \R$ the function
\[
u_\e(x_1,x_2):=\begin{cases}
|x_2|&\text{ if }|x_2|\geq \e |x_1|\\
\frac{\e |x_1|}{2} + \frac{x_2^2}{2\e |x_1|}&\text{ if }|x_2|<\e |x_1|
\end{cases}.
\]

Then $u_\e\in W^{2,\infty}_\loc(\R^2\setminus\{0\})$ is convex and one-homogeneous, with $\mu_{u_\e} = \frac{4\e}{3}\delta_0$.

In addition, $u_\e$ is linear in the two cones $\{x\in\R^2\,:\,x_2>\e |x_1|\}$ and $\{x\in \R^2\,:\,x_2<-\e|x_1|\}$.
\end{lemma}

\begin{proof}
We check that for $x\neq 0$
\[
\nabla u(x) = \begin{cases}
  e_2 & \text{ if }x_2>\e |x_1|\\
  \begin{pmatrix} \frac{\e}{2} - \frac{x_2^2}{2\e x_1^2}\\ \frac{x_2}{\e x_1}\end{pmatrix}&\text{ if }|x_2|\leq \e|x_1|\\ 
  -e_2 & \text{ if }x_2<\e |x_1|
\end{cases}
  \]

  In particular, $\nabla u$ is locally Lipschitz in $\R^2\setminus\{0\}$. A direct calculation shows that $u(x) = \max_{y\in\R^2\setminus \{0\}} x\cdot \nabla u(y)$, showing that $u$ is convex.

  Since $\nabla u$ is $0$-homogeneous, $\mu_u = \sigma \delta_0$ for some $\sigma\geq 0$. We calculate
\[
\sigma = \L^2(\conv\{\nabla u(y)\,:\,y\in\R^2\setminus\{0\}\}) = 2\int_{-1}^1 \frac{\e}{2}(1-t^2)\,\d t = \frac{4\e}{3}.  
\]
\end{proof}

\begin{proof}[Proof of Remark \ref{rem:nobc}]
First we construct a solution $v_1$ of $\mu_{v_1} = \overline\mu$  that is in $W^{2,\infty}_{\mathrm{loc}}(\Omega\setminus\{a_1,\dots,a_N\})$. Define for $\e>0$ the horizontal double cone $K_\e:= \{x\in\R^2\,:\,|x_2|\leq \e |x_1|\}$. 

Choose $e\in S^1$ and $\e>0$ depending on $\Omega$ and $a_1,\ldots,a_N$ such that the sets 
\[
K_i:=\Omega\cap (a_i+R_e K_\e), \quad i=1,\dots,N\,,
\]
have mutually positive distance from each other, where $R_e\in SO(2)$ is the rotation mapping $e_1$ to $e$. See Figure \ref{fig:nobc} for an illustration. Then define
\[
v_1(x) := \sum_{i=1}^N \sqrt{\frac{3\sigma_i}{4\e}} u_\e(R_e^T(x-a_i))\, ,
\]
where $u_\e:\R^2\to\R$ is the function from Lemma \ref{lem:nobcaux}. By our choice of $e,\e$, it follows that $\mu_{v_1} = \overline\mu$. Note here that additivity of the Monge-Amp\`ere measure holds only since $\nabla^2 u_\e = 0$ outside of $K_\e$. Hence $\curl\curl(\nabla v_1\otimes\nabla v_1-\nabla v_0\otimes \nabla v_0)=0$ in $W^{-2,2}(\Omega)$. By \cite[Theorem 7.2]{amrouche2006characterizations}\footnote{In \cite{amrouche2006characterizations}  the statement is made for $\R^3$ valued functions, but the relevant proofs in this paper  work out verbatim for $\R^2$ valued functions.} this implies the existence of 
\[
\bu_0\in W^{1,2}(\Omega;\R^2)
\]
such that
\[
\nabla v_1\otimes\nabla v_1-\nabla v_0\otimes \nabla v_0=-(\nabla \bu_0+\nabla \bu_0^T)\,.
\]
Now we define $v_h$ on $B(a_i,h)$ by requiring it to be convex, in $W^{2,\infty}(\Omega)$ with $|\nabla^2v_h|\leq Ch^{-1}$ and 
\[
v_h=v_1, \quad\nabla v_h=\nabla v_1 \quad\text{ on }\partial B(a_i,h)\,.
\] 
Setting $v_h=v_1$ on $\Omega\setminus \bigcup_{i=1}^NB(a_i,h)$ and $\bu_h=\bu_0$ completes our construction. The estimate $E_h(\bu_h,v_h)\leq Ch^2\log\frac{1}{h}$ is the result of a straightforward calculation.
\end{proof}

\begin{figure}[h]
  \centering
  \includegraphics[height=5cm]{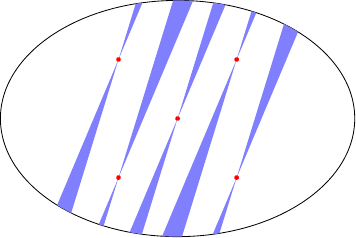}
\caption{Construction of the convex function $v_1$ in the proof of Remark \ref{rem:nobc}. The double cones $K_1,\ldots,K_5$ are shown in blue. Their orientation and width are chosen so as to make them pairwise disjoint, thus ensuring that $\mu_{v_1} = \overline\mu$. \label{fig:nobc}}. \end{figure}



\bibliographystyle{alpha}
\bibliography{connect}

\end{document}